\documentclass[12pt]{article}
\usepackage{amsmath,amsfonts,amsthm,amscd,upref,amstext}
\usepackage[dvips]{graphicx}
\usepackage{subfigure}
\usepackage{pb-diagram,pb-xy}
\usepackage[all]{xy}
\usepackage[usenames, dvipsnames]{color}
\usepackage[english]{babel}
\usepackage[utf8]{inputenc}
\usepackage{amsmath}
\usepackage{graphicx}
\usepackage[colorinlistoftodos]{todonotes}

\newtheorem{prop}{Proposition}[section]
\newtheorem{thm}[prop]{Theorem}
\newtheorem{defn}[prop]{Definition}

\newtheorem{cor}[prop]{Corollary}

\newtheorem{rem}[prop]{Remark}
\newtheorem{lem}[prop]{Lemma}

\numberwithin{equation}{section}

\begin{document}

\title{Local Homology with Respect to a Pair of Ideals}
\author{V.H. Jorge P\'erez$^{1,}\,$\thanks{Universidade de S{\~a}o Paulo -
ICMC, Caixa Postal 668, 13560-970, S{\~a}o Carlos-SP, Brazil ({\it
e-mail: vhjperez@icmc.usp.br}) 2000 Mathematics Subject Classification: 13D45; 16E30; 13J99. {\it Key words}: Inverse limit; Local homology; Local cohomology; Linearly compact module.}\,\,\,\,and\,\,\,C.H. Tognon$^{2,}$ \thanks{Universidade de S{\~a}o Paulo -
ICMC, Caixa Postal 668, 13560-970, S{\~a}o Carlos-SP, Brazil ({\it
e-mail: tognon@icmc.usp.br}).}}

\date{}
\maketitle

\vspace{0.3cm}
\begin{abstract}
We introduce a generalization of the notion of local homology module, which we call a local homology module with respect to a pair of ideals $\left(I,J\right)$, and study its various properties such as vanishing, co-support and co-associated. We also discuss its connection with ordinary local homology.
\end{abstract}

\maketitle
\section{Introduction}
Throughout this paper, $R$ is a commutative ring with non-zero identity. For a $R$-module $M$ and a ideal $I$ of ring $R$ there are two important functors in commutative algebra and algebraic geometry which are the $I$-torsion functor $\Gamma_{I}\left(\bullet\right)$ and the $I$-adic completion functor $\Lambda_{I}\left(\bullet\right)$ defined by $\Gamma_{I}\left(M\right) = \bigcup_{t > 0} \left(0:_{M} I^{t}\right)$ and $\Lambda_{I}\left(M\right) = \varprojlim_{t \in \mathbb{N}} M/I^{t}M$. It should be noted that the $I$-torsion functor $\Gamma_{I}\left(\bullet\right)$ is left exact and its $i$th right derived functor ${\rm H}^{i}_{I}\left(\bullet\right)$ is called the $i$th local cohomology functor with respect to $I$. However, the $I$-adic completion functor $\Lambda_{I}\left(\bullet\right)$ is neither right nor left exact, so computing its left derived functors is in general difficult.

The local cohomology theory of Grothendieck has proved to be an important tool in algebraic geometry, commutative algebra and algebraic topology. Its dual theory of local homology is also studied by many mathematicians: Greenlees and May \cite{may}, Tarr\'io \cite{tarrio}, and Cuong and Nam \cite{completamento I-adico}, etc. In \cite{grothendieck}, we have that Grothendieck introduced the definition of local cohomology module. Let $\mathfrak{a}$ be an ideal of $R$, and $M$ be an $R$-module, then the module ${\rm H}^{i}_{\mathfrak{a}}\left(M\right) = \varinjlim_{t \in \mathbb{N}}{\rm Ext}^{i}_{R}\left(R/\mathfrak{a}^{t},M\right)$ is called the $i$-th local cohomology module of $M$ with respect to $\mathfrak{a}$.

In this paper, we define the local homology module with respect to a pair of ideals $\left(I,J\right)$ which is in some sense dual to \cite{cohomologia local de um par de ideais} local cohomology defined by a pair of ideals. We have studied local homology module defined by a pair of ideals for linearly compact modules. It should be mentioned that the class of linearly compact modules is great, it contains important classes of modules in algebra. For example, artinian modules are linearly compact and discrete \cite[$3.10$]{macdonald}. Moreover, if $R$ is a complete local Noetherian ring and $M$ is a finitely generated $R$-module, then $M$ is semidiscrete (that means every submodules of M is closed) and linearly compact \cite[$7.3$]{macdonald}.

Section $2$ is devoted to recall some definitions, results obtained for the functor of local homology $\Lambda_{I,J}\left(-\right)$ defined by a pair of ideals $\left(I,J\right)$ and to prove results for the local homology module ${\rm H}_{i}^{I,J}\left(M\right)$ defined by a pair of ideals when $M$ is a linearly compact $R$-module. In section $3$ put some results on vanishing of local homology module with respect to a pair of ideals.

In the last section, we define an $I$-stable module as a module $M$ such that for each element $x \in I$ there is a positive integer $n$ such that $x^{t}M = x^{n}M$ for all $t \geq n$. The Theorem \ref{sequencia exata de linearmente compacto} provides a long exact sequence in local homology modules with respect to the pair of ideals $\left(I,J\right)$, and to finish we have the Theorem \ref{finitude}, which shows that, being $i$ be a non-negative integer, if ${\rm H}_{j}^{I,J}\left(M\right)$ is $\mathfrak{a}$-stable, for all $j < i$, and $G$ is a closed $R$-submodule of ${\rm H}_{i}^{I,J}\left(M\right)$ such that ${\rm H}_{i}^{I,J}\left(M\right)/G$ is $\mathfrak{a}$-stable then the set ${\rm Coass}_{R}\left(G\right)$ is finite where $R$ is a Noetherian ring and $M$ an $\mathfrak{a}$-stable semidiscrete linearly compact $R$-module.

\section{Definition and basic properties}

\begin{defn}\label{funtor} {\rm (\cite[Section $2$]{completamento I-adico}) Let $R$ be a ring (it is not assumed here that $R$ is Noetherian) and $I$ an ideal of $R$. Let $M$ be an $R$-module. Consider the inverse system of $R$-modules $\left\{M/I^{t}M\right\}_{t \in \mathbb{N}}$ with natural epimorphisms
\begin{center}
$\pi_{k,t}: M/I^{k}M \longrightarrow M/I^{t}M$, for all $0 < t \leq k$, $t,k \in \mathbb{N}$.
\end{center}
We use $\Lambda_{I}\left(M\right) = \varprojlim_{t \in \mathbb{N}} M/I^{t}M$ to denote the $I$-adic completion of $M$. It is known that the functor of the $I$-adic completion $\Lambda_{I}\left(-\right)$ is an additive covariant functor from the category of $R$-modules and $R$-homomorphisms to itself. We denote by $\mathfrak{L}_{i}^{I}\left(M\right)$ the $i$th left derived module of $\Lambda_{I}\left(M\right)$. Since the tensor functor is not left exact and the inverse limit is not right exact on the category of $R$-modules, the functor $\Lambda_{I}\left(-\right)$ is neither left nor right exact.}
\end{defn}

\begin{defn}\label{funtor de dois ideais} {\rm (\cite[Definition $3.1$]{cohomologia local de um par de ideais}) Let $\tilde{W}\left(I,J\right)$ denote the set of ideals $\mathfrak{a}$ of ring $R$ such that $I^{n}\subseteq \mathfrak{a} + J$ for some integer $n \geq 1$, i.e.,
\begin{center}
$\tilde{W}\left(I,J\right) := \left\{\mathfrak{a}\subset R \,|\, \mathfrak{a} \ \textrm{is ideal of} \ R \ \textrm{and} \ I^{n}\subseteq \mathfrak{a} + J \ \textrm{for some} \ n \in \mathbb{N}\right\}$.
\end{center}
We define a partial order on $\tilde{W}\left(I,J\right)$ by letting $\mathfrak{a} \leq \mathfrak{b}$ if and only if $\mathfrak{b}\subseteq \mathfrak{a}$, for $\mathfrak{a}, \mathfrak{b} \in \tilde{W}\left(I,J\right)$. With this relation of partial order we have that $\tilde{W}\left(I,J\right)$ is a directed set. Moreover, the family of $R$-modules $\left\{\Lambda_{\mathfrak{a}}\left(M\right)\right\}_{\mathfrak{a} \in \tilde{W}\left(I,J\right)}$ is an inverse system of $R$-modules. We define then,
\begin{center}
$\Lambda_{I,J}\left(M\right) := \varprojlim_{\mathfrak{a} \in \tilde{W}\left(I,J\right)} \Lambda_{\mathfrak{a}}\left(M\right) = \varprojlim_{\mathfrak{a} \in \tilde{W}\left(I,J\right)} \varprojlim_{n \in \mathbb{N}} M/\mathfrak{a}^{n}M$.
\end{center}}
\end{defn}

\begin{rem} {\rm $\left(i\right)$ For $\mathfrak{a} \in \tilde{W}\left(I,J\right)$, since the functor $\Lambda_{\mathfrak{a}}\left(-\right)$ is neither left nor right exact and the inverse limit is not right exact on the category of $R$-modules, we have that the functor $\Lambda_{I,J}\left(-\right)$ is neither left nor right exact. Moreover, the functor of the $\mathfrak{a}$-adic completion $\Lambda_{\mathfrak{a}}\left(-\right)$ is an covariant functor and $\varprojlim_{\mathfrak{a} \in \tilde{W}\left(I,J\right)}$ is also an covariant functor; thus, $\Lambda_{I,J}\left(-\right)$ is an covariant functor from the category of $R$-modules and $R$-homomorphisms to itself.\newline
$\left(ii\right)$ We denote by $\mathfrak{L}_{i}^{I,J}\left(M\right)$ the $i$th left derived module of $\Lambda_{I,J}\left(M\right)$. So, as $\Lambda_{I,J}\left(-\right)$ is neither left nor right exact, it follows that, in general, $\mathfrak{L}_{0}^{I,J} \neq \Lambda_{I,J}$. However, $\mathfrak{L}_{0}^{I,J}$ is a right exact functor and its left derived functors for $i > 0$ are the same as those of $\Lambda_{I,J}$.}
\end{rem}

Now, let $M$ be a $R$-module and $N$ be a submodule of $M$. For $m \in M$, we define a subset of $M$: $m + N = \left\{m + n \;|\; n \in N\right\}$. A subset of $M$ is said to be a coset of $N$ if there exists $m \in M$ such that it is equal to $m + N$ \cite[Definition $5$]{cosets}. Moreover, $x \in m + N$ if and only if there exists $n$ such that $n \in N$ and $x = m + n$.

Consider that the ring $R$ is Noetherian and has a topological structure. Let us recall the concept of linearly compact modules by terminology of Macdonald \cite[Definition $3.1$]{macdonald}. Let $M$ be a topological $R$-module. A nucleus of $M$ is a neighbourhood of the zero element of $M$, and a nuclear base of $M$ is a base for the nuclei of $M$. $M$ is Hausdorff if and only if the intersection of all the nuclei of M is $0$. It is said to be linearly topologized if $M$ has a nuclear base $\Sigma$ consisting of submodules. A Hausdorff linearly topologized $R$-module $M$ is said to be linearly compact if $M$ has the following property: if $\mathfrak{F}$ is a family of closed cosets (i.e., cosets of closed submodules) in $M$ which has the finite intersection property, then the cosets in $\mathfrak{F}$ have a non-empty intersection. It should be noted that an Artinian $R$-module is linearly compact with the discrete topology \cite[Theorem $2.1$]{modulos linearmente compactos sobre aneis noetherianos}. A Hausdorff linearly topologized $R$-module $M$ is called {\it semidiscrete} if every submodule of $M$ is closed. The class of semidiscrete linearly compact modules contains all Artinian modules. For an $R$-module $M$ and a submodule $N \subseteq M$ we define the set $\left(N:_{M} \mathfrak{a}\right) = \left\{m \in M \;|\; \mathfrak{a}m \subseteq N\right\}$. Observe that $\left(N:_{M} \mathfrak{a}\right)$ is a submodule of $M$ and that $N \subseteq \left(N:_{M} \mathfrak{a}\right)$. For an $R$-module $M$, the $\mathfrak{a}$-torsion of $M$ is defined by
\begin{center}
$\Gamma_{\mathfrak{a}}\left(M\right) := \bigcup_{n \in \mathbb{N}} \left(0:_{M} \mathfrak{a}^{n}\right) = \left\{m \in M \;|\; \mathfrak{a}^{n}m = 0, \textrm{for some integer} \ n \geq 1\right\}$.
\end{center}
Observe that $\Gamma_{\mathfrak{a}}\left(M\right)$ is a submodule of $M$.

\begin{defn}\label{categoria nova} {\rm Denote by $\mathfrak{C}_{\Lambda_{I,J}}\left(R\right)$ the category of all $R$-modules $M$ (and homomorphisms) such that $\mathfrak{L}_{0}^{I,J}\left(M\right) \cong \Lambda_{I,J}\left(M\right)$.}
\end{defn}

{\rm The following result is a generalization, to the case of a pair of ideals, of \cite[Corollary $2.5$]{left derived functors}.}

\begin{prop}\label{homomorfismo sobrejetivo} Let $f: M\rightarrow N$ be a homomorphism of $R$-modules, where $M$ is a linearly compact $R$-module. If $f$ is surjective, then the homomorphism $\Lambda_{I,J}\left(f\right): \Lambda_{I,J}\left(M\right)\rightarrow \Lambda_{I,J}\left(N\right)$ is also surjective.
\end{prop}

\begin{proof}
We have the epimorphism $f: M\rightarrow N$. By \cite[Corollary $2.5$]{left derived functors} we have that, for all $\mathfrak{a} \in \tilde{W}\left(I,J\right)$, the homomorphism $\Lambda_{\mathfrak{a}}\left(f\right): \Lambda_{\mathfrak{a}}\left(M\right)\rightarrow \Lambda_{\mathfrak{a}}\left(N\right)$ is surjective. As $M$ is linearly compact $R$-module, it follows that $\left\{M/\mathfrak{a}^{t}M\right\}_{t \in \mathbb{N}}$ is a inverse system of linearly compact $R$-modules. By \cite[Lemma $2.3$, item $\left(iv\right)$]{uma teoria de homologia local para modulo linearmente compacto} it follows that $\varprojlim_{t \in \mathbb{N}} M/\mathfrak{a}^{t}M = \Lambda_{\mathfrak{a}}\left(M\right)$ is linearly compact $R$-module. Thus, $\left\{{\rm Ker}\left(\Lambda_{\mathfrak{a}}\left(f\right)\right)\right\}_{\mathfrak{a} \in \tilde{W}\left(I,J\right)}$ is also a inverse system of linearly compact $R$-modules. Therefore, we have
\begin{center}
$0\rightarrow \left\{{\rm Ker}\left(\Lambda_{\mathfrak{a}}\left(f\right)\right)\right\}_{\mathfrak{a} \in \tilde{W}\left(I,J\right)}\rightarrow \left\{\Lambda_{\mathfrak{a}}\left(M\right)\right\}_{\mathfrak{a} \in \tilde{W}\left(I,J\right)}\rightarrow \left\{\Lambda_{\mathfrak{a}}\left(N\right)\right\}_{\mathfrak{a} \in \tilde{W}\left(I,J\right)}\rightarrow 0$,
\end{center}
a short exact sequence of inverse systems of $R$-modules. By \cite[Lemma $2.4$]{uma teoria de homologia local para modulo linearmente compacto} we have that the sequence of inverse limits
\begin{center}
$0\rightarrow \varprojlim_{\mathfrak{a} \in \tilde{W}\left(I,J\right)} {\bf L}_{\mathfrak{a}}\rightarrow \varprojlim_{\mathfrak{a} \in \tilde{W}\left(I,J\right)} \Lambda_{\mathfrak{a}}\left(M\right)\rightarrow \varprojlim_{\mathfrak{a} \in \tilde{W}\left(I,J\right)} \Lambda_{\mathfrak{a}}\left(N\right)\rightarrow 0$,
\end{center}
is exact, where ${\bf L}_{\mathfrak{a}} := {\rm Ker}\left(\Lambda_{\mathfrak{a}}\left(f\right)\right)$. Therefore, the homomorphism $\Lambda_{I,J}\left(f\right): \Lambda_{I,J}\left(M\right)\rightarrow \Lambda_{I,J}\left(N\right)$ is also surjective, as required.
\end{proof}

\begin{lem}\label{limite inverso comuta} Let $M$ be a linearly compact $R$-module and let $\mathfrak{a} \in \tilde{W}\left(I,J\right)$ an ideal any of $R$. Then we have the following isomorphism:
\begin{center}
$\varprojlim_{\mathfrak{a} \in \tilde{W}\left(I,J\right)} \mathfrak{L}_{i}^{\mathfrak{a}}\left(M\right) \cong \mathfrak{L}_{i}^{I,J}\left(M\right)$, for all $i \geq 0$.
\end{center}
\end{lem}

\begin{proof}
Let ${\bf F}_{\bullet}: \ldots \rightarrow F_{i}\rightarrow \ldots \rightarrow F_{1}\rightarrow F_{0}\rightarrow M\rightarrow 0$ be a free resolution of $M$. Since $M$ is linearly compact, we have that $\left\{R/\mathfrak{a}^{t}\right\}_{t \in \mathbb{N}}$ is a inverse system of linearly compact $R$-modules, for all $\mathfrak{a} \in \tilde{W}\left(I,J\right)$; so we have by \cite[Lemma $2.3$, item $\left(iii\right)$]{uma teoria de homologia local para modulo linearmente compacto} that $\left\{F_{i} \otimes_{R} R/\mathfrak{a}^{t}\right\}_{t \in \mathbb{N}}$ is a inverse system of linearly compact $R$-modules, for all $i \geq 0$. Now, by \cite[Lemma $2.3$, item $\left(iv\right)$]{uma teoria de homologia local para modulo linearmente compacto}, we have that $\left\{\varprojlim_{t \in \mathbb{N}} F_{i}/\mathfrak{a}^{t}F_{i}\right\}_{\mathfrak{a} \in \tilde{W}\left(I,J\right)}$ is a inverse system of linearly compact $R$-modules. Moreover, we have that $\varprojlim_{\mathfrak{a} \in \tilde{W}\left(I,J\right)}$ is a covariant additive exact functor on linearly compact $R$-modules, by \cite[Lemma $2.4$]{uma teoria de homologia local para modulo linearmente compacto}. Thus, we have that
\begin{center}
${\rm H}_{i}\left(\varprojlim_{\mathfrak{a} \in \tilde{W}\left(I,J\right)}\left(\varprojlim_{t \in \mathbb{N}} {\bf F}_{\bullet} \otimes_{R} R/\mathfrak{a}^{t}\right)\right) \cong \varprojlim_{\mathfrak{a} \in \tilde{W}\left(I,J\right)} {\rm H}_{i}\left(\varprojlim_{t \in \mathbb{N}} {\bf F}_{\bullet} \otimes_{R} R/\mathfrak{a}^{t}\right)$,
\end{center}
by \cite[$6.1$, Theorem $1$]{northcott}. Therefore, $\varprojlim_{\mathfrak{a} \in \tilde{W}\left(I,J\right)} \mathfrak{L}_{i}^{\mathfrak{a}}\left(M\right) \cong \mathfrak{L}_{i}^{I,J}\left(M\right)$, for all $i \geq 0$, as required.
\end{proof}

\begin{thm}\label{pertence a categoria nova} Let $M$ be a linearly compact $R$-module such that the inverse system of $R$-modules $\left\{\mathfrak{a}^{t}M\right\}_{t \in \mathbb{N}}$, for all $\mathfrak{a} \in \tilde{W}\left(I,J\right)$, is stationary, i.e., there is a positive integer $n$ such that $\mathfrak{a}^{t}M = \mathfrak{a}^{n}M$, for all $t \geq n$. Then, $M \in \mathfrak{C}_{\Lambda_{I,J}}\left(R\right)$.
\end{thm}

\begin{proof}
From of \cite[Theorem $2.3$]{completamento I-adico} we have that $\mathfrak{L}_{0}^{\mathfrak{a}}\left(M\right) \cong \Lambda_{\mathfrak{a}}\left(M\right)$, for all $\mathfrak{a} \in \tilde{W}\left(I,J\right)$. Therefore, we have that $\varprojlim_{\mathfrak{a} \in \tilde{W}\left(I,J\right)} \mathfrak{L}_{0}^{\mathfrak{a}}\left(M\right) \cong \Lambda_{I,J}\left(M\right)$. Now, by Lemma \ref{limite inverso comuta}, it follows that $\varprojlim_{\mathfrak{a} \in \tilde{W}\left(I,J\right)} \mathfrak{L}_{0}^{\mathfrak{a}}\left(M\right) \cong \mathfrak{L}_{0}^{I,J}\left(M\right)$ and so $\mathfrak{L}_{0}^{I,J}\left(M\right) \cong \Lambda_{I,J}\left(M\right)$. Thus, we have by definition that, $M \in \mathfrak{C}_{\Lambda_{I,J}}\left(R\right)$, as required.
\end{proof}

Artinian modules certainly satisfy the hypothesis of Theorem \ref{pertence a categoria nova}. Therefore, we have the following immediate consequence.

\begin{cor} Let $M$ be an Artinian $R$-module. Then, we have that the homomorphism of $R$-modules
\begin{center}
$\phi_{M}: \mathfrak{L}_{0}^{I,J}\left(M\right) \rightarrow \Lambda_{I,J}\left(M\right)$,
\end{center}
is an isomorphism.
\end{cor}

\begin{proof}
Indeed, as $M$ is an Artinian $R$-module, we have that for all $\mathfrak{a} \in \tilde{W}\left(I,J\right)$ the family of $R$-modules $\left\{\mathfrak{a}^{t}M\right\}_{t \in \mathbb{N}}$, that is a family of $R$-submodules of $M$, is stationary. Moreover, $M$ is a linearly compact $R$-module with the discrete topology. Therefore, by the Theorem \ref{pertence a categoria nova}, we have that $\mathfrak{L}_{0}^{I,J}\left(M\right)$ is isomorphic to $\Lambda_{I,J}\left(M\right)$. Therefore, $\phi_{M}$ is an isomorphism, as required.
\end{proof}

We also have the following consequence.

\begin{cor} Let $M$ be an linearly compact $R$-module. Then the following statements:
\begin{itemize}
\item[$\left(i\right)$] $\mathfrak{a}M = M$, for all $\mathfrak{a} \in \tilde{W}\left(I,J\right)$;
\item[$\left(ii\right)$] $\mathfrak{L}_{0}^{I,J}\left(M\right) = 0$;
\item[$\left(iii\right)$] $\Lambda_{I,J}\left(M\right) = 0$,
\end{itemize}
are such that we have the implications: $\left(i\right) \Rightarrow \left(ii\right)$; $\left(i\right)$ and $\left(ii\right) \Rightarrow \left(iii\right)$.
\end{cor}

\begin{proof}
$\left(i\right) \Rightarrow \left(ii\right)$: From the hypothesis we have that $M = \mathfrak{a}^{t}M$, for all $t \in \mathbb{N}$ and for all $\mathfrak{a} \in \tilde{W}\left(I,J\right)$. Thus, the inverse system of $R$-modules $\left\{\mathfrak{a}^{t}M\right\}_{t \in \mathbb{N}}$ is stationary. Therefore, by the Theorem \ref{pertence a categoria nova}, we have that $\mathfrak{L}_{0}^{I,J}\left(M\right) \cong \Lambda_{I,J}\left(M\right) = \varprojlim_{\mathfrak{a} \in \tilde{W}\left(I,J\right)} \varprojlim_{t \in \mathbb{N}} M/\mathfrak{a}^{t}M = 0$, as required.\newline
$\left(i\right)$ and $\left(ii\right) \Rightarrow \left(iii\right)$: By $\left(i\right)$ we have that $\left\{\mathfrak{a}^{t}M\right\}_{t \in \mathbb{N}}$ is stationary; by the hypothesis $\left(ii\right)$ we have that $\mathfrak{L}_{0}^{I,J}\left(M\right) = 0$; by the Theorem \ref{pertence a categoria nova} we have that $\mathfrak{L}_{0}^{I,J}\left(M\right) \cong \Lambda_{I,J}\left(M\right)$, and then it follows that $\Lambda_{I,J}\left(M\right) = 0$.
\end{proof}

{\rm The following result is a generalization, to the case of a pair of ideals, of \cite[Lemma $2.6$]{left derived functors}.}

\begin{prop} Let $G\rightarrow L\stackrel{f}{\rightarrow} N\rightarrow 0$ be an exact sequence of $R$-modules, where $G$ and $L$ are linearly compact $R$-modules, and such that the inverse system $\left\{\mathfrak{a}^{t}N\right\}_{t \in \mathbb{N}}$ is stationary, for all $\mathfrak{a} \in \tilde{W}\left(I,J\right)$. Then the following induced sequence
\begin{center}
$\Lambda_{I,J}\left(G\right)\rightarrow \Lambda_{I,J}\left(L\right)\rightarrow \Lambda_{I,J}\left(N\right)\rightarrow 0$,
\end{center}
is exact.
\end{prop}

\begin{proof}
Set $K = {\rm Ker}\left(f\right)$; we have induced exact sequences
\begin{center}
$G\rightarrow K\rightarrow 0$ and $0\rightarrow K\rightarrow L\rightarrow N\rightarrow 0$.
\end{center}
From of Proposition \ref{homomorfismo sobrejetivo}, the first exact sequence induces an exact sequence
\begin{center}
$\Lambda_{I,J}\left(G\right)\rightarrow \Lambda_{I,J}\left(K\right)\rightarrow 0$.
\end{center}
As in the proof of \cite[Theorem $2.3$]{completamento I-adico}, the second exact sequence gives an exact sequence
\begin{center}
$\Lambda_{\mathfrak{a}}\left(K\right)\rightarrow \Lambda_{\mathfrak{a}}\left(L\right)\rightarrow \Lambda_{\mathfrak{a}}\left(N\right)\rightarrow 0$,
\end{center}
for all $\mathfrak{a} \in \tilde{W}\left(I,J\right)$. Since $K$ is linearly compact $R$-module, because $L$ is linearly compact, we have that $\left\{K/\mathfrak{a}^{t}K\right\}_{t \in \mathbb{N}}$ is a inverse system of linearly compact $R$-modules; by \cite[Lemma $2.3$, item $\left(iv\right)$]{uma teoria de homologia local para modulo linearmente compacto} it follows that $\Lambda_{\mathfrak{a}}\left(K\right)$ is linearly compact $R$-module. Therefore, $\left\{\Lambda_{\mathfrak{a}}\left(K\right)\right\}_{\mathfrak{a} \in \tilde{W}\left(I,J\right)}$ is a inverse system of linearly compact modules. By \cite[Lemma $2.4$]{uma teoria de homologia local para modulo linearmente compacto} we have that the sequence of inverse limits
\begin{center}
$\Lambda_{I,J}\left(K\right)\rightarrow \Lambda_{I,J}\left(L\right)\rightarrow \Lambda_{I,J}\left(N\right)\rightarrow 0$,
\end{center}
is exact. Therefore, the induced sequence
\begin{center}
$\Lambda_{I,J}\left(G\right)\rightarrow \Lambda_{I,J}\left(L\right)\rightarrow \Lambda_{I,J}\left(N\right)\rightarrow 0$,
\end{center}
is exact, as required.
\end{proof}

\begin{defn} {\rm By \cite[Theorem $3.2$]{cohomologia local de um par de ideais}, we have that for an $R$-module $M$  there is a natural isomorphism
\begin{center}
${\rm H}^{i}_{I,J}\left(M\right) \cong \varinjlim_{\mathfrak{a} \in \tilde{W}\left(I,J\right)} {\rm H}^{i}_{\mathfrak{a}}\left(M\right)$
\end{center}
for any integer $i \geq 0$.}
\end{defn}

This suggests the following definition: Let $I, J$ be an ideals of $R$ and $M$ an $R$-module. The $i$th local homology module ${\rm H}_{i}^{I,J}\left(M\right)$ of $M$ with respect to the pair of ideals $\left(I,J\right)$ is defined by
\begin{center}
${\rm H}_{i}^{I,J}\left(M\right) = \varprojlim_{\mathfrak{a} \in \tilde{W}\left(I,J\right)} {\rm H}_{i}^{\mathfrak{a}}\left(M\right) = \varprojlim_{\mathfrak{a} \in \tilde{W}\left(I,J\right)} \varprojlim_{t \in \mathbb{N}} {\rm Tor}_{i}^{R}\left(R/\mathfrak{a}^{t},M\right)$.
\end{center}

\begin{rem} {\rm In the above definition, we have immediately that when $i = 0$, ${\rm H}_{0}^{I,J}\left(M\right) = \Lambda_{I,J}\left(M\right)$. Moreover, for all $i \geq 0$, we have that when $J = 0$, ${\rm H}_{i}^{I,0}\left(M\right)$ coincides with the $i$th local homology module ${\rm H}_{i}^{I}\left(M\right)$ of $M$ with respect to $I$ (\cite[Definition $3.1$]{completamento I-adico}).}
\end{rem}

{\rm Denote by $\varprojlim_{t \in \mathbb{N}}^{i}$ the $i$th right derived functor of the inverse limit $\varprojlim_{t \in \mathbb{N}}$. If $\left\{M_{t}\right\}_{t \in \mathbb{N}}$ is an inverse system of linearly compact $R$-modules with continuous homomorphisms, then $\varprojlim_{t \in \mathbb{N}}^{1} M_{t} = 0$, by \cite[Theorem $7.1$]{les foncteurs derives de inverse limit}.}

\begin{prop}\label{dual} Let $M$ be an $R$-module with $R$ an Noetherian ring. Then the following statements are true.
\begin{itemize}
\item[$\left(i\right)$] For all $i \geq 0$ and for all $\mathfrak{a} \in \tilde{W}\left(I,J\right)$, the local homology module ${\rm H}_{i}^{I,J}\left(M\right)$ with respect to the pair of ideals $\left(I,J\right)$ is $\mathfrak{a}$-separated, i.e.,
\begin{center}
$\bigcap_{s > 0} \mathfrak{a}^{s}{\rm H}_{i}^{I,J}\left(M\right) = 0$, for all $\mathfrak{a} \in \tilde{W}\left(I,J\right)$.
\end{center}
\item[$\left(ii\right)$] Suppose that $\left(R,\mathfrak{m}\right)$ is a local ring. Then for all $i \geq 0$,
\begin{center}
${\rm H}_{i}^{I,J}\left({\rm D}\left(M\right)\right) \cong {\rm D}\left({\rm H}^{i}_{I,J}\left(M\right)\right)$,
\end{center}
where ${\rm D}\left(M\right) = {\rm Hom}_{R}\left(M,E\right)$ is the Matlis dual module of $M$ and $E = {\rm E}\left(R/\mathfrak{m}\right)$ is the injective envelope of the residue field $R/\mathfrak{m}$.
\end{itemize}
\end{prop}

\begin{proof}
$\left(i\right)$ Note first that for any inverse system of $R$-modules $\left\{M_{t}\right\}_{t \in \mathbb{N}}$, we have that
\begin{center}
$\mathfrak{a} \varprojlim_{t \in \mathbb{N}} M_{t}\subseteq \varprojlim_{t \in \mathbb{N}} \mathfrak{a}M_{t}$, for all $\mathfrak{a} \in \tilde{W}\left(I,J\right)$.
\end{center}
Thus,
\begin{center}
$\bigcap_{s > 0} \mathfrak{a}^{s}{\rm H}_{i}^{I,J}\left(M\right) \cong \varprojlim_{s \in \mathbb{N}} \mathfrak{a}^{s} \varprojlim_{\mathfrak{a} \in \tilde{W}\left(I,J\right)} {\rm H}_{i}^{\mathfrak{a}}\left(M\right)$,
\end{center}
and by initial observation we have that
\begin{center}
$\varprojlim_{s \in \mathbb{N}} \mathfrak{a}^{s} \varprojlim_{\mathfrak{a} \in \tilde{W}\left(I,J\right)} {\rm H}_{i}^{\mathfrak{a}}\left(M\right)\subseteq \varprojlim_{s \in \mathbb{N}} \varprojlim_{\mathfrak{a} \in \tilde{W}\left(I,J\right)} \mathfrak{a}^{s}{\rm H}_{i}^{\mathfrak{a}}\left(M\right)$,
\end{center}
where ${\rm H}_{i}^{\mathfrak{a}}\left(M\right) = \varprojlim_{t \in \mathbb{N}} {\rm Tor}_{i}^{R}\left(R/\mathfrak{a}^{t},M\right)$. Therefore,
\begin{center}
$\bigcap_{s > 0} \mathfrak{a}^{s}{\rm H}_{i}^{I,J}\left(M\right)\subseteq \varprojlim_{s \in \mathbb{N}} \varprojlim_{\mathfrak{a} \in \tilde{W}\left(I,J\right)} \varprojlim_{t \in \mathbb{N}} \mathfrak{a}^{s}{\rm Tor}_{i}^{R}\left(R/\mathfrak{a}^{t},M\right)$.
\end{center}
As any two inverse limits commute \cite[Theorem $2.26$]{rotman} we have that
$$
\begin{array}{lll}
\bigcap_{s > 0} \mathfrak{a}^{s}{\rm H}_{i}^{I,J}\left(M\right)&\subseteq &\varprojlim_{\mathfrak{a} \in \tilde{W}\left(I,J\right)} \varprojlim_{t \in \mathbb{N}} \varprojlim_{s \in \mathbb{N}} \mathfrak{a}^{s}{\rm Tor}_{i}^{R}\left(R/\mathfrak{a}^{t},M\right)\\& =& 0,
\end{array}
$$
since $\mathfrak{a}^{s}{\rm Tor}_{i}^{R}\left(R/\mathfrak{a}^{t},M\right) = 0$, for all $s \geq t$. Thus, $\bigcap_{s > 0} \mathfrak{a}^{s}{\rm H}_{i}^{I,J}\left(M\right) = 0$.\newline
$\left(ii\right)$ We have that for a direct system of $R$-modules $\left\{N_{t}\right\}_{t \in \mathbb{N}}$,
\begin{center}
$\varprojlim_{t \in \mathbb{N}} {\rm D}\left(N_{t}\right)\cong {\rm D}\left(\varinjlim_{t \in \mathbb{N}} N_{t}\right)$, by \cite[Theorem $2.27$]{rotman} and
\end{center}
\begin{center}
${\rm Tor}_{i}^{R}\left(R/\mathfrak{a}^{t},{\rm D}\left(M\right)\right) = {\rm D}\left({\rm Ext}^{i}_{R}\left(R/\mathfrak{a}^{t},M\right)\right)$ for all $\mathfrak{a} \in \tilde{W}\left(I,J\right)$
\end{center}
by \cite[Proposition 3.4.14 (ii)]{strooker}. Thus,
$$
\begin{array}{lll}
{\rm H}_{i}^{I,J}\left({\rm D}\left(M\right)\right) &\cong &\varprojlim_{\mathfrak{a} \in \tilde{W}\left(I,J\right)} {\rm H}_{i}^{\mathfrak{a}}\left({\rm D}\left(M\right)\right)\\ & =& \varprojlim_{\mathfrak{a} \in \tilde{W}\left(I,J\right)} \varprojlim_{t \in \mathbb{N}} {\rm Tor}_{i}^{R}\left(R/\mathfrak{a}^{t},{\rm D}\left(M\right)\right)
\end{array}
$$
and so we have that,
\begin{center}
${\rm H}_{i}^{I,J}\left({\rm D}\left(M\right)\right) \cong \varprojlim_{\mathfrak{a} \in \tilde{W}\left(I,J\right)} \varprojlim_{t \in \mathbb{N}} {\rm D}\left({\rm Ext}^{i}_{R}\left(R/\mathfrak{a}^{t},M\right)\right)$.
\end{center}
Therefore,
$$
\begin{array}{lll}
{\rm H}_{i}^{I,J}\left({\rm D}\left(M\right)\right)& \cong & \varprojlim_{\mathfrak{a} \in \tilde{W}\left(I,J\right)} {\rm D}\left(\varinjlim_{t \in \mathbb{N}}{\rm Ext}^{i}_{R}\left(R/\mathfrak{a}^{t},M\right)\right)\\& \cong & \varprojlim_{\mathfrak{a} \in \tilde{W}\left(I,J\right)} {\rm D}\left({\rm H}^{i}_{\mathfrak{a}}\left(M\right)\right).
\end{array}
$$
Thus,
\begin{center}
${\rm H}_{i}^{I,J}\left({\rm D}\left(M\right)\right) \cong {\rm D}\left(\varinjlim_{\mathfrak{a} \in \tilde{W}\left(I,J\right)} {\rm H}^{i}_{\mathfrak{a}}\left(M\right)\right) \cong {\rm D}\left({\rm H}^{i}_{I,J}\left(M\right)\right)$,
\end{center}
where, $\varinjlim_{\mathfrak{a} \in \tilde{W}\left(I,J\right)} {\rm H}^{i}_{\mathfrak{a}}\left(M\right) \cong {\rm H}^{i}_{I,J}\left(M\right)$, for all $i \geq 0$, by \cite[Theorem $3.2$]{cohomologia local de um par de ideais}. Thus, ${\rm H}_{i}^{I,J}\left({\rm D}\left(M\right)\right) \cong {\rm D}\left({\rm H}^{i}_{I,J}\left(M\right)\right)$.
\end{proof}

{\rm The following result is a generalization, to the case of a pair of ideals, of \cite[Corollary 2.3]{a generalization of local homology functors}.}

\begin{cor}\label{Corollary2.3} Let $M$ be an $R$-module where $R$ is a local ring and Noetherian. Then we have,
\begin{itemize}
\item[$\left(i\right)$] ${\rm H}^{i}_{I,J}\left(M\right) = 0$ if and only if ${\rm H}_{i}^{I,J}\left({\rm D}(M)\right) = 0$.
\item[$\left(ii\right)$] If $M$ is an Artinian $R$-module, then ${\rm H}_i^{I,J}\left(M\right)=0$ if and only if ${\rm H}^{i}_{I,J}\left({\rm D}(M)\right) = 0$.
\end{itemize}
\end{cor}

\begin{proof}
It is well know that for a $R$-module $N$ we have that, $N = 0$ if and only if ${\rm D}\left(N\right) = 0$. The corollary is now immediate from \cite[Theorem $1.6$,$\left(5\right)$]{Ooishi}.
\end{proof}

\begin{prop}\label{comutatividade} Let $\left\{M_{s}\right\}_{s \in \mathbb{N}}$ be an inverse system of linearly compact $R$-modules with the continuous homomorphisms. Then,
\begin{center}
${\rm H}_{i}^{I,J}\left(\varprojlim_{s \in \mathbb{N}} M_{s}\right) \cong \varprojlim_{s \in \mathbb{N}} {\rm H}_{i}^{I,J}\left(M_{s}\right)$.
\end{center}
\end{prop}

\begin{proof} By definition we have that
\begin{center}
${\rm H}_{i}^{I,J}\left(\varprojlim_{s \in \mathbb{N}} M_{s}\right) = \varprojlim_{\mathfrak{a} \in \tilde{W}\left(I,J\right)} {\rm H}_{i}^{\mathfrak{a}}\left(\varprojlim_{s \in \mathbb{N}} M_{s}\right)$.
\end{center}
Now, note that
\begin{center}
${\rm H}_{i}^{\mathfrak{a}}\left(\varprojlim_{s \in \mathbb{N}} M_{s}\right) = \varprojlim_{t \in \mathbb{N}} {\rm Tor}_{i}^{R}\left(R/\mathfrak{a}^{t},\varprojlim_{s \in \mathbb{N}} M_{s}\right)$,
\end{center}
and by \cite[Lemma $2.7$]{uma teoria de homologia local para modulo linearmente compacto} we have that:
\begin{center}
${\rm Tor}_{i}^{R}\left(R/\mathfrak{a}^{t},\varprojlim_{s \in \mathbb{N}} M_{s}\right) \cong \varprojlim_{s \in \mathbb{N}} {\rm Tor}_{i}^{R}\left(R/\mathfrak{a}^{t},M_{s}\right)$.
\end{center}
Therefore,
\begin{center}
${\rm H}_{i}^{\mathfrak{a}}\left(\varprojlim_{s \in \mathbb{N}} M_{s}\right) \cong \varprojlim_{t \in \mathbb{N}} \varprojlim_{s \in \mathbb{N}} {\rm Tor}_{i}^{R}\left(R/\mathfrak{a}^{t},M_{s}\right)$.
\end{center}
Since by \cite[Theorem $2.26$]{rotman} inverse limits are commuted, we have that
\begin{center}
${\rm H}_{i}^{\mathfrak{a}}\left(\varprojlim_{s \in \mathbb{N}} M_{s}\right) \cong \varprojlim_{s \in \mathbb{N}} \varprojlim_{t \in \mathbb{N}} {\rm Tor}_{i}^{R}\left(R/\mathfrak{a}^{t},M_{s}\right) = \varprojlim_{s \in \mathbb{N}} {\rm H}_{i}^{\mathfrak{a}}\left(M_{s}\right)$.
\end{center}
Therefore,
\begin{center}
${\rm H}_{i}^{I,J}\left(\varprojlim_{s \in \mathbb{N}} M_{s}\right) = \varprojlim_{\mathfrak{a} \in \tilde{W}\left(I,J\right)} \varprojlim_{s \in \mathbb{N}} {\rm H}_{i}^{\mathfrak{a}}\left(M_{s}\right)$,
\end{center}
and then
\begin{center}
${\rm H}_{i}^{I,J}\left(\varprojlim_{s \in \mathbb{N}} M_{s}\right) \cong \varprojlim_{s \in \mathbb{N}} \varprojlim_{\mathfrak{a} \in \tilde{W}\left(I,J\right)} {\rm H}_{i}^{\mathfrak{a}}\left(M_{s}\right) = \varprojlim_{s \in \mathbb{N}} {\rm H}_{i}^{I,J}\left(M_{s}\right)$,
\end{center}
as required.
\end{proof}

In the following theorem, we assume that $f: R\rightarrow R^{'}$ is a homomorphism of rings. Also, for an ideal $I$ of $R$, we denote its extension to $R^{'}$ by $I^{e}$.

\begin{thm}\label{extensao}Let $R$ be a Noetherian ring. Let $M$ be a linearly compact $R^{'}$-module with $\left(R^{'},\mathfrak{m}\right)$ local Noetherian ring. Furthermore, let $f: R\rightarrow R^{'}$ be a ring homomorphism such that $f\left(J\right) = JR^{'} = J^{e}$. Then we have the following isomorphism of $R^{'}$-modules
\begin{center}
${\rm H}_{i}^{I,J}\left(M\right) \cong {\rm H}_{i}^{I^{e},J^{e}}\left(M\right)$
\end{center}
for all $0 \leq i \in \mathbb{Z}$.
\end{thm}

\begin{proof}
We first prove in the special case $M$ is an Artinian $R^{'}$-module. In this case, by \cite[Theorem $1.6$, item $\left(5\right)$]{Ooishi}, we have that ${\rm D}\left({\rm D}\left(M\right)\right) \cong M$. Therefore, ${\rm H}_{i}^{I,J}\left(M\right) \cong {\rm H}_{i}^{I,J}\left({\rm D}\left({\rm D}\left(M\right)\right)\right) \cong {\rm D}\left({\rm H}^{i}_{I,J}\left({\rm D}\left(M\right)\right)\right)$, by Proposition \ref{dual}, item $\left(ii\right)$. Now, by \cite[Theorem $2.7$]{cohomologia local de um par de ideais}, we have that ${\rm H}^{i}_{I,J}\left({\rm D}\left(M\right)\right) \cong {\rm H}^{i}_{I^{e},J^{e}}\left({\rm D}\left(M\right)\right)$. Thus, ${\rm H}_{i}^{I,J}\left(M\right) \cong {\rm D}\left({\rm H}^{i}_{I^{e},J^{e}}\left({\rm D}\left(M\right)\right)\right) \cong {\rm H}_{i}^{I^{e},J^{e}}\left(M\right)$, by Proposition \ref{dual}, item $\left(ii\right)$, for all $i \geq 0$.

Let $M$ be a linearly compact $R^{'}$-module. In this case denote by $\mathfrak{M}$ a nuclear base of $M$. It follows from \cite[Property $4.7$]{macdonald} that $M = \varprojlim_{U \in \mathfrak{M}} \left(M/U\right)$, where the modules $M/U$ ($U \in \mathfrak{M}$) are Artinian $R$-modules. In virtue of Proposition \ref{comutatividade} we have that
\begin{center}
${\rm H}_{i}^{I,J}\left(M\right) = {\rm H}_{i}^{I,J}\left(\varprojlim_{U \in \mathfrak{M}} \left(M/U\right)\right) \cong \varprojlim_{U \in \mathfrak{M}} {\rm H}_{i}^{I,J}\left(M/U\right)$
\end{center}
e
\begin{center}
${\rm H}_{i}^{I^{e},J^{e}}\left(M\right) = {\rm H}_{i}^{I^{e},J^{e}}\left(\varprojlim_{U \in \mathfrak{M}} \left(M/U\right)\right) \cong \varprojlim_{U \in \mathfrak{M}} {\rm H}_{i}^{I^{e},J^{e}}\left(M/U\right)$.
\end{center}
By our claim above ${\rm H}_{i}^{I,J}\left(M/U\right) \cong {\rm H}_{i}^{I^{e},J^{e}}\left(M/U\right)$ for all $i \geq 0$. Applying then $\varprojlim_{U \in \mathfrak{M}}$ we obtain that ${\rm H}_{i}^{I,J}\left(M\right) \cong {\rm H}_{i}^{I^{e},J^{e}}\left(M\right)$, for all $i \geq 0$, as required.
\end{proof}


In the next result for $N$ a $R$-module we denote by $\hat{N}$ the $\mathfrak{m}$-adic completion of $N$. Moreover, note that if $N$ is an Artinian module over a local ring $R$, then $N$ has a natural structure as an Artinian module over $\hat{R}$, according to \cite[$1.11$]{sharp}.

\begin{cor}\label{completo} Let $\left(R,\mathfrak{m}\right)$ be a local Noetherian ring with the $\mathfrak{m}$-adic topology, $\phi: R\rightarrow \hat{R}$ be a ring homomorphism and $M$ an Artinian $R$-module. Suppose that $\phi$ satisfies the equality $\phi\left(J\right) = J\hat{R} = \hat{J}$. Then
\begin{center}
${\rm H}_{i}^{I,J}\left(M\right) \cong {\rm H}_{i}^{\hat{I},\hat{J}}\left(M\right)$ \ \ \ \ (as $R$-modules)
\end{center}
for all $0 \leq i \in \mathbb{Z}$.
\end{cor}

\begin{proof}
Since $M$ is an Artinian $R$-module, by \cite[Theorem $1.6$, item $\left(5\right)$]{Ooishi}, we have that ${\rm D}\left({\rm D}\left(M\right)\right) \cong M$. Therefore, ${\rm H}_{i}^{I,J}\left(M\right) \cong {\rm H}_{i}^{I,J}\left({\rm D}\left({\rm D}\left(M\right)\right)\right) \cong {\rm D}\left({\rm H}^{i}_{I,J}\left({\rm D}\left(M\right)\right)\right)$, by Proposition \ref{dual}, item $\left(ii\right)$. Thus, by \cite[Theorem $2.7$]{cohomologia local de um par de ideais}, we have that ${\rm H}^{i}_{I,J}\left({\rm D}\left(M\right)\right) \cong {\rm H}^{i}_{\hat{I},\hat{J}}\left({\rm D}\left(M\right)\right)$ as $R$-modules. Therefore, ${\rm H}_{i}^{I,J}\left(M\right) \cong {\rm D}\left({\rm H}^{i}_{\hat{I},\hat{J}}\left({\rm D}\left(M\right)\right)\right) \cong {\rm H}_{i}^{\hat{I},\hat{J}}\left(M\right)$, by Proposition \ref{dual}, item $\left(ii\right)$ for all $i \geq 0$ as $R$-modules, as required.
\end{proof}

\begin{rem}\label{observacao de ryo takahashi}{\rm (\cite[Example of the Theorem 2.7]{cohomologia local de um par de ideais})  Here we should remark that the hypothesis $\phi\left(J\right) = JR^{'}$ in the Theorem \ref{extensao} and Corollary \ref{completo} is necessary. Let $k$ be a field, $R = k\left[x,y\right]$ and $R^{'} = k\left[x,y,z\right]/\left(xz - yz^{2}\right)$. Set $I = xR$ and $J = yR$. For a natural ring homomorphism $\phi$ from $R$ to $R^{'}$, we have $\phi\left(J\right)\subset JR^{'}$ but $\phi\left(J\right) \neq JR^{'}$ and ${\rm H}^{0}_{I,J}\left(R^{'}\right) \neq {\rm H}^{0}_{IR^{'},JR^{'}}\left(R^{'}\right)$. Thus, we have that  ${\rm D}\left({\rm H}^{0}_{I,J}\left(R^{'}\right)\right) \neq {\rm D}\left({\rm H}^{0}_{IR^{'},JR^{'}}\left(R^{'}\right)\right)$ and then, by Proposition \ref{dual}, item $\left(ii\right)$, it follows that ${\rm H}_{0}^{I,J}\left({\rm D}\left(R^{'}\right)\right) \neq {\rm H}_{0}^{IR^{'},JR^{'}}\left({\rm D}\left(R^{'}\right)\right)$.}
\end{rem}

Let $M$ be a linearly compact $R$-module. Then, for all $\mathfrak{a} \in \tilde{W}\left(I,J\right)$, ${\rm Tor}_{i}^{R}\left(R/\mathfrak{a}^{t},M\right)$ is also a linearly compact $R$-module by the topology defined as in \cite[Lemma $2.6$]{uma teoria de homologia local para modulo linearmente compacto}; so we have an induced topology on the local homology module ${\rm H}_{i}^{I,J}\left(M\right)$.

\begin{thm}\label{linearmente compacto} Let $M$ be a linearly compact $R$-module. Then for all $i \geq 0$ we have that ${\rm H}_{i}^{I,J}\left(M\right)$ is a linearly compact $R$-module.
\end{thm}

\begin{proof} For all $\mathfrak{a} \in \tilde{W}\left(I,J\right)$, we have by \cite[Lemma $2.6$]{uma teoria de homologia local para modulo linearmente compacto} that the family of $R$-modules $\left\{{\rm Tor}_{i}^{R}\left(R/\mathfrak{a}^{t},M\right)\right\}_{t \in \mathbb{N}}$ is a family that forms an inverse system of linearly compact $R$-modules with continuous homomorphisms. By the \cite[Lemma $2.3$, item $\left(iv\right)$]{uma teoria de homologia local para modulo linearmente compacto} we have that
\begin{center}
$\varprojlim_{t \in \mathbb{N}} {\rm Tor}_{i}^{R}\left(R/\mathfrak{a}^{t},M\right) = {\rm H}_{i}^{\mathfrak{a}}\left(M\right)$,
\end{center}
is also a linearly compact $R$-module. Now, the family of $R$-modules of local homology $\left\{{\rm H}_{i}^{\mathfrak{a}}\left(M\right)\right\}_{\mathfrak{a} \in \tilde{W}\left(I,J\right)}$ forms an inverse system of linearly compact $R$-modules with continuous homomorphisms. Again by \cite[Lemma $2.3$, item $\left(iv\right)$]{uma teoria de homologia local para modulo linearmente compacto}, we have that
\begin{center}
$\varprojlim_{\mathfrak{a} \in \tilde{W}\left(I,J\right)} {\rm H}_{i}^{\mathfrak{a}}\left(M\right) = {\rm H}_{i}^{I,J}\left(M\right)$,
\end{center}
is linearly compact $R$-module.
\end{proof}

\begin{rem}{\rm Let $R$ be an Artinian ring and $M$ a finitely generated $R$-module. Then, we have that ${\rm H}_{0}^{I,J}\left(M\right)$ is a linearly compact $R$-module.}
\end{rem}

\begin{thm}\label{isomorfo a funtor} Let $R$ be a ring and $M$ be a linearly compact $R$-module. Then, we have that ${\rm H}_{i}^{I,J}\left(M\right) \cong \mathfrak{L}_{i}^{I,J}\left(M\right)$ for all $i \geq 0$.
\end{thm}

\begin{proof} By \cite[Theorem $3.6$]{left derived functors}, we have the isomorphism ${\rm H}_{i}^{\mathfrak{a}}\left(M\right) \cong \mathfrak{L}_{i}^{\mathfrak{a}}\left(M\right)$ for all $i \geq 0$ and for all $\mathfrak{a} \in \tilde{W}\left(I,J\right)$. Thus,
\begin{center}
$\varprojlim_{\mathfrak{a} \in \tilde{W}\left(I,J\right)} {\rm H}_{i}^{\mathfrak{a}}\left(M\right) \cong \varprojlim_{\mathfrak{a} \in \tilde{W}\left(I,J\right)} \mathfrak{L}_{i}^{\mathfrak{a}}\left(M\right) \cong \mathfrak{L}_{i}^{I,J}\left(M\right)$,
\end{center}
by Lemma \ref{limite inverso comuta}. Therefore, ${\rm H}_{i}^{I,J}\left(M\right) \cong \mathfrak{L}_{i}^{I,J}\left(M\right)$ for all $i \geq 0$, as required.
\end{proof}

\begin{thm}\label{sequencia exata de linearmente compacto} Let $R$ be a Noetherian ring and
$0\rightarrow M^{'}\rightarrow M\rightarrow M^{''}\rightarrow 0$
a short exact sequence of linearly compact $R$-modules. Then we have a long exact sequence of the local homology modules
\begin{flushleft}
$\ldots \rightarrow {\rm H}_{i}^{I,J}\left(M^{'}\right)\rightarrow {\rm H}_{i}^{I,J}\left(M\right)\rightarrow {\rm H}_{i}^{I,J}\left(M^{''}\right)\rightarrow \ldots \rightarrow {\rm H}_{0}^{I,J}\left(M^{'}\right)\rightarrow {\rm H}_{0}^{I,J}\left(M\right)\rightarrow {\rm H}_{0}^{I,J}\left(M^{''}\right)\rightarrow 0$.
\end{flushleft}
Moreover, each module this sequence is linearly compact.
\end{thm}

\begin{proof} The short exact sequence of $R$-modules
\begin{center}
$0\rightarrow M^{'}\rightarrow M\rightarrow M^{''}\rightarrow 0$
\end{center}
gives rise to a long exact sequence, by \cite[Theorem $8.3$]{rotman pequeno}
\begin{center}
$\ldots \rightarrow {\rm Tor}_{i}^{R}\left(R/\mathfrak{a}^{t},M^{'}\right)\rightarrow {\rm Tor}_{i}^{R}\left(R/\mathfrak{a}^{t},M\right)\rightarrow {\rm Tor}_{i}^{R}\left(R/\mathfrak{a}^{t},M^{''}\right)\rightarrow \ldots \rightarrow {\rm Tor}_{0}^{R}\left(R/\mathfrak{a}^{t},M^{'}\right)\rightarrow {\rm Tor}_{0}^{R}\left(R/\mathfrak{a}^{t},M\right)\rightarrow {\rm Tor}_{0}^{R}\left(R/\mathfrak{a}^{t},M^{''}\right)\rightarrow 0$
\end{center}
for any $\mathfrak{a} \in \tilde{W}\left(I,J\right)$. As $R/\mathfrak{a}^{t}$ is finitely generated $R$-module and $M^{'}$, $M$, $M^{''}$ are linearly compact $R$-modules, we have by \cite[Lemma $2.8$]{co-associados finitos} that for all $i \geq 0$, ${\rm Tor}_{i}^{R}\left(R/\mathfrak{a}^{t},M^{'}\right)$, ${\rm Tor}_{i}^{R}\left(R/\mathfrak{a}^{t},M\right)$ and ${\rm Tor}_{i}^{R}\left(R/\mathfrak{a}^{t},M^{''}\right)$, are linearly compact $R$-modules. Now, by \cite[Lemma $2.4$]{uma teoria de homologia local para modulo linearmente compacto}, we have that $\varprojlim_{t \in \mathbb{N}}$ is exact functor on linearly compact $R$-modules. Thus, applying $\varprojlim_{t \in \mathbb{N}}$ the previous long exact sequence, we have that
\begin{flushleft}
$\ldots \rightarrow {\rm H}_{i}^{\mathfrak{a}}\left(M^{'}\right)\rightarrow {\rm H}_{i}^{\mathfrak{a}}\left(M\right)\rightarrow {\rm H}_{i}^{\mathfrak{a}}\left(M^{''}\right)\rightarrow \ldots \rightarrow {\rm H}_{0}^{\mathfrak{a}}\left(M^{'}\right)\rightarrow {\rm H}_{0}^{\mathfrak{a}}\left(M\right)\rightarrow {\rm H}_{0}^{\mathfrak{a}}\left(M^{''}\right)\rightarrow 0$
\end{flushleft}
is exact sequence. As $M^{'}$, $M$ and $M^{''}$ are linearly compact $R$-modules, by \cite[Proposition $3.3$]{uma teoria de homologia local para modulo linearmente compacto}, the modules in previous long exact sequence are linearly compact $R$-modules. Then, applying $\varprojlim_{\mathfrak{a} \in \tilde{W}\left(I,J\right)}$ the previous sequence, we obtain the long exact sequence:
\begin{flushleft}
$\ldots \rightarrow {\rm H}_{i}^{I,J}\left(M^{'}\right)\rightarrow {\rm H}_{i}^{I,J}\left(M\right)\rightarrow {\rm H}_{i}^{I,J}\left(M^{''}\right)\rightarrow \ldots \rightarrow {\rm H}_{0}^{I,J}\left(M^{'}\right)\rightarrow {\rm H}_{0}^{I,J}\left(M\right)\rightarrow {\rm H}_{0}^{I,J}\left(M^{''}\right)\rightarrow 0$.
\end{flushleft}
By the Theorem \ref{linearmente compacto}, we have that each module this sequence is linearly compact $R$-module, as required.
\end{proof}

\begin{rem} {\rm According to Theorem \ref{isomorfo a funtor} we have, by the Theorem \ref{sequencia exata de linearmente compacto}, that we obtain a long exact sequence for $\mathfrak{L}_{i}^{I,J}\left(M\right)$ for all $i \geq 0$ the $i$th left derived module of $\Lambda_{I,J}\left(M\right)$.}
\end{rem}

\begin{cor}\label{sequencia exata de artinianos} Let $0\rightarrow M^{'}\rightarrow M\rightarrow M^{''}\rightarrow 0$ be a short exact sequence of Artinian $R$-modules, where $R$ is Noetherian ring. Then, for all $\mathfrak{a} \in \tilde{W}\left(I,J\right)$, we have a long exact sequence of modules of local homology with respect to the pair of ideals $\left(I,J\right)$
\begin{center}
$\ldots \rightarrow {\rm H}_{i}^{I,J}\left(M^{'}\right)\rightarrow {\rm H}_{i}^{I,J}\left(M\right)\rightarrow {\rm H}_{i}^{I,J}\left(M^{''}\right)\rightarrow \ldots \rightarrow {\rm H}_{0}^{I,J}\left(M^{'}\right)\rightarrow {\rm H}_{0}^{I,J}\left(M\right)\rightarrow {\rm H}_{0}^{I,J}\left(M^{''}\right)\rightarrow 0$.
\end{center}
\end{cor}

\begin{proof}
Since Artinian modules are linearly compact modules with the discrete topology \cite[Theorem $2.1$]{modulos linearmente compactos sobre aneis noetherianos}, the result it follows from Theorem \ref{sequencia exata de linearmente compacto}.
\end{proof}

\begin{thm}\label{teorema de generalization of local homology functors} Let $\left(R,\mathfrak{m}\right)$ be a Noetherian local ring, with a unique maximal ideal $\mathfrak{m}$ and $M$ be a Artinian $R$-module. Then, for a positive integer $s$, the following statements are equivalent:
\begin{itemize}
\item[$\left(i\right)$] ${\rm H}_{i}^{I,J}\left(M\right)$ is Artinian $R$-module, for all $i < s$, $i \geq 0$;
\item[$\left(ii\right)$] $\mathfrak{a}\subseteq {\rm Rad}\left({\rm Ann}_{R}\left({\rm H}_{i}^{I,J}\left(M\right)\right)\right)$, for all $i < s$, $i \geq 0$ and for all $\mathfrak{a} \in \tilde{W}\left(I,J\right)$.
\end{itemize}
\end{thm}

\begin{proof}
$\left(i\right) \Rightarrow \left(ii\right)$: Suppose that $i < s$, $i \geq 1$. Since ${\rm H}_{i}^{I,J}\left(M\right)$ is Artinian $R$-module for all $i < s$, there exists a positive integer $n$ such that $\mathfrak{a}^{t}{\rm H}_{i}^{I,J}\left(M\right) = \mathfrak{a}^{n}{\rm H}_{i}^{I,J}\left(M\right)$, for all $t \geq n$. Therefore,
\begin{center}
$\mathfrak{a}^{n}{\rm H}_{i}^{I,J}\left(M\right) = \bigcap_{t > 0} \mathfrak{a}^{t}{\rm H}_{i}^{I,J}\left(M\right) = 0$, since ${\rm H}_{i}^{I,J}\left(M\right)$ is $\mathfrak{a}$-separated.
\end{center}
Thus, $\mathfrak{a}\subseteq {\rm Rad}\left({\rm Ann}_{R}\left({\rm H}_{i}^{I,J}\left(M\right)\right)\right)$, for all $i < s$, $i \geq 1$ and for all $\mathfrak{a} \in \tilde{W}\left(I,J\right)$.\newline
$\left(ii\right) \Rightarrow \left(i\right)$: We use induction on $s$. When $s = 1$, we must show that ${\rm H}_{0}^{I,J}\left(M\right)$ is Artinian $R$-module. Since $M$ is Artinian there is a positive integer $m$ such that $\mathfrak{a}^{t}M = \mathfrak{a}^{m}M$, for all $t \geq m$. Then ${\rm H}_{0}^{\mathfrak{a}}\left(M\right) = R/\mathfrak{a}^{m} \otimes_{R} M$. Since $R/\mathfrak{a}^{m} \otimes_{R} M$ is Artinian \cite[Proposition $2.13$]{Ooishi} (since $M$ is Artinian $R$-module), we have by the proof of \cite[Theorem $3.4$]{a generalization of local homology functors} that ${\rm H}_{0}^{\mathfrak{a}}\left(M\right)$ is Artinian $R$-module. We have then that $\left\{{\rm H}_{0}^{\mathfrak{a}}\left(M\right)\right\}_{\mathfrak{a} \in \tilde{W}\left(I,J\right)}$ is an inverse system of Artinian $R$-modules; therefore, ${\rm H}_{0}^{I,J}\left(M\right)$ is Artinian $R$-module, by \cite[Properties $2.4$, $3.6$ and $3.3$]{macdonald} and \cite[Properties $27.2$ and $27.3$]{lefschetz}. Suppose that $s > 1$. By \cite[Theorem $3.3$]{a generalization of local homology functors} we can replace $M$ by $\bigcap_{t > 0} \mathfrak{a}^{t}M$. Since $M$ is Artinian the last module is just equal to $\mathfrak{a}^{n}M$, for sufficiently large $n$. Therefore, we may assume that $\mathfrak{a}M = M$. Since $M$ is Artinian, there is an element $x \in \mathfrak{a}$ such that $xM = M$ (\cite[Proposition $1.1$, item $\left(i\right)$]{torsion theory co-cohen macaulay}). Thus, by the hypothesis, there exists positive integer $t$ such that $x^{t}{\rm H}_{i}^{I,J}\left(M\right) = 0$ for all $i < s$, $i \geq 1$. Then the short exact sequence
\begin{center}
$0\rightarrow \left(0 :_{M} x^{t}\right)\rightarrow M\stackrel{x^{t}}{\rightarrow} M\rightarrow 0$
\end{center}
provides us a long exact sequence, by Corollary \ref{sequencia exata de artinianos}:
\begin{center}
$0\rightarrow {\rm H}_{i + 1}^{I,J}\left(M\right)\rightarrow {\rm H}_{i}^{I,J}\left(\left(0 :_{M} x^{t}\right)\right)\rightarrow {\rm H}_{i}^{I,J}\left(M\right)\rightarrow 0$
\end{center}
for all $i < s - 1$. It follows that $\mathfrak{a}\subseteq {\rm Rad}\left({\rm Ann}_{R}\left({\rm H}_{i}^{I,J}\left(0 :_{M} x^{t}\right)\right)\right)$, and by inductive hypothesis we have that ${\rm H}_{i}^{I,J}\left(\left(0 :_{M} x^{t}\right)\right)$ is Artinian, for all $i < s - 1$. Thus, ${\rm H}_{i}^{I,J}\left(M\right)$ is Artinian $R$-module for all $i < s$. This finishes the inductive step.
\end{proof}

\begin{rem} {\rm We note that in the implication $\left(i\right)\Rightarrow \left(ii\right)$ in the proof of Theorem \ref{teorema de generalization of local homology functors} we need not assume that $M$ is Artinian $R$-module.}
\end{rem}

\section{Vanishing Results}

\begin{defn} {\rm (\cite[Definition $1.1$]{cohomologia local de um par de ideais}) For an $R$-module $M$ we denote by $\Gamma_{I,J}\left(M\right)$ the set of elements $x$ of $M$ such that $I^{n}x\subseteq Jx$ for some integer $n \geq 1$, i.e.,
\begin{center}
$\Gamma_{I,J}\left(M\right) = \left\{x \in M \;|\; I^{n}x\subseteq Jx \ \textrm{for some integer} \ n \geq 1\right\}$.
\end{center}
We say that $M$ is $\left(I,J\right)$-torsion (respectively $\left(I,J\right)$-torsion-free) precisely when $\Gamma_{I,J}\left(M\right) = M$ (respectively $\Gamma_{I,J}\left(M\right) = 0$). Note that when we have the ideal $J = 0$, we say that $M$ is $I$-torsion when $\Gamma_{I}\left(M\right) = 0$, and we say that $M$ is $I$-torsion-free when $\Gamma_{I}\left(M\right) = 0$.}
\end{defn}

We now recall the concept of noetherian dimension of an $R$-module $M$, denoted by ${\rm Ndim}\left(M\right)$. This notion was first introduced by R.N. Roberts \cite[Definitions]{roberts} by the name Krull dimension. Later, Kirby \cite[Definitions]{kirby} changed this terminology of Roberts and refereed to noetherian dimension to avoid confusion with well-known Krull dimension of finitely generated modules. Let $M$ be an $R$-module. When $M = 0$ we put ${\rm Ndim}\left(M\right) = -1$. Then by induction, for any ordinal $\alpha$, we put ${\rm Ndim}\left(M\right) = \alpha$ when
\begin{itemize}
\item[$\left(i\right)$] ${\rm Ndim}\left(M\right) < \alpha$ is false; and
\item[$\left(ii\right)$] for every ascending chain $M_{0}\subseteq M_{1}\subseteq \ldots$ of submodules of $M$, there exists a positive integer $m_{0}$ such that ${\rm Ndim}\left(M_{m + 1}/M_{m}\right) < \alpha$, for all $m \geq m_{0}$.
\end{itemize}
Thus $M$ is non-zero and Noetherian if and only if ${\rm Ndim}\left(M\right) = 0$.

Recall that a module $M$ is simple if it is non-zero and does not admit a proper non-zero submodule. Simplicity of a module $M$ is equivalent to say that $Rm = M$, for every $m$ non-zero in $M$. The ${\rm Soc}\left(M\right)$ the socle of $M$ is the sum of all simple submodules of $M$, i.e., is the submodule
\begin{center}
${\rm Soc}\left(M\right) = \sum \left\{N \;|\; N \ \textrm{is simple submodule of} \ M\right\}$.
\end{center}
We recall also that a module $M$ is said to be semisimple if it satisfies any of the equivalent conditions:
\begin{itemize}
\item[$\left(i\right)$] it is a sum of simple submodules.
\item[$\left(ii\right)$] it is a direct sum of simple submodules.
\end{itemize}
So the socle of $M$ is the largest submodule of $M$ generated by simple modules, or equivalently, it is the largest semisimple submodule of $M$.

{\rm The following result is a generalization, to the case of a pair of ideals, of \cite[Theorem $3.16$]{left derived functors}.}

\begin{thm} Let $\left(R,\mathfrak{m}\right)$ be a local ring with the $\mathfrak{m}$-adic topology such that $R$ is Noetherian. Let $M$ a linearly compact $R$-module. Then ${\rm H}_{i}^{\mathfrak{m},J}\left(M\right) = 0$ for all $i > \dim\left(R\right)$.
\end{thm}

\begin{proof}
We first prove in the special case $M$ is an Artinian $R$-module. From Corollary \ref{completo}, we may assume in this case that $\left(R,\mathfrak{m}\right)$ is a complete ring. By \cite[Theorem $10.2.12$, item $\left(iii\right)$]{brodmann} we have that ${\rm D}\left({\rm D}\left(M\right)\right) \cong M$ and ${\rm D}\left(M\right)$ is a Noetherian $R$-module. Thus, we have ${\rm D}\left(M\right)$ a finitely generated $R$-module. By the Proposition \ref{dual}, item $\left(ii\right)$, we have that
\begin{center}
${\rm D}\left({\rm H}^{i}_{\mathfrak{m},J}\left({\rm D}\left(M\right)\right)\right) \cong {\rm H}_{i}^{\mathfrak{m},J}\left({\rm D}\left({\rm D}\left(M\right)\right)\right) \cong {\rm H}_{i}^{\mathfrak{m},J}\left(M\right)$.
\end{center}
Now, by \cite[Theorem $3.2$]{duality and vanishing of generalized local cohomology}, we have that ${\rm H}^{i}_{\mathfrak{a}}\left({\rm D}\left(M\right)\right) = 0$ for all $i > \dim\left(R\right)$ and for all $\mathfrak{a} \in \tilde{W}\left(I,J\right)$. By \cite[Theorem $3.2$]{cohomologia local de um par de ideais} it follows that
\begin{center}
${\rm H}^{i}_{\mathfrak{m},J}\left({\rm D}\left(M\right)\right) \cong \varinjlim_{\mathfrak{a} \in \tilde{W}\left(\mathfrak{m},J\right)} {\rm H}^{i}_{\mathfrak{a}}\left({\rm D}\left(M\right)\right)$.
\end{center}
Therefore, ${\rm H}^{i}_{\mathfrak{m},J}\left({\rm D}\left(M\right)\right) = 0$ for all $i > \dim\left(R\right)$. Thus, by the Corollary \ref{Corollary2.3}, item $\left(i\right)$, it follows that ${\rm D}\left({\rm H}^{i}_{\mathfrak{m},J}\left({\rm D}\left(M\right)\right)\right) = 0$ for all $i > \dim\left(R\right)$. Hence, ${\rm H}_{i}^{\mathfrak{m},J}\left(M\right) = 0$ for all $i > \dim\left(R\right)$.

Let $M$ be a linearly compact $R$-module. In this case, denote by $\mathfrak{M}$ a nuclear base of $M$. It follows from \cite[Property $4.7$]{macdonald} that $M = \varprojlim_{U \in \mathfrak{M}} \left(M/U\right)$, where the modules $M/U$ ($U \in \mathfrak{M}$) are Artinian $R$-modules. In virtue of Proposition \ref{comutatividade} we have that
\begin{center}
${\rm H}_{i}^{\mathfrak{m},J}\left(M\right) = {\rm H}_{i}^{\mathfrak{m},J}\left(\varprojlim_{U \in \mathfrak{M}} \left(M/U\right)\right) \cong \varprojlim_{U \in \mathfrak{M}} {\rm H}_{i}^{\mathfrak{m},J}\left(M/U\right)$.
\end{center}
By the part initial we have that ${\rm H}_{i}^{\mathfrak{m},J}\left(M/U\right) = 0$ for all $i > \dim\left(R\right)$ and for all $U \in \mathfrak{M}$. Therefore, ${\rm H}_{i}^{\mathfrak{m},J}\left(M\right) = 0$ for all $i > \dim\left(R\right)$. We conclude the proof.
\end{proof}

{\rm The following result is a generalization, to the case of a pair of ideals, of \cite[Lemma $4.2$]{uma teoria de homologia local para modulo linearmente compacto}.}

\begin{prop} Let $M$ be a semidiscrete linearly compact $R$-module where $R$ is Noetherian ring, and with ${\rm Soc}\left(M\right) = 0$. Then ${\rm H}_{i}^{I,J}\left(M\right) = 0$ for all $i > 0$.
\end{prop}

\begin{proof}
The proof will be in two parts. Consider first that $M$ is an Artinian $R$-module. By \cite[Lemma $4.2$]{uma teoria de homologia local para modulo linearmente compacto} we have an isomorphism $M \stackrel{x}{\rightarrow} M$ for some $x \in \mathfrak{a}$ and for all $\mathfrak{a} \in \tilde{W}\left(I,J\right)$; it induces an isomorphism ${\rm H}_{i}^{\mathfrak{a}}\left(M\right)\stackrel{x}{\rightarrow} {\rm H}_{i}^{\mathfrak{a}}\left(M\right)$ for all $i > 0$ and for all $\mathfrak{a} \in \tilde{W}\left(I,J\right)$; since, by definition, we have that ${\rm H}_{i}^{I,J}\left(M\right) = \varprojlim_{\mathfrak{a} \in \tilde{W}\left(I,J\right)} {\rm H}_{i}^{\mathfrak{a}}\left(M\right)$ it follows that we have an isomorphism
\begin{center}
${\rm H}_{i}^{I,J}\left(M\right)\stackrel{x}{\rightarrow} {\rm H}_{i}^{I,J}\left(M\right)$ for all $i > 0$.
\end{center}
By Proposition \ref{dual} item $\left(i\right)$, we have
\begin{center}
${\rm H}_{i}^{I,J}\left(M\right) \cong x{\rm H}_{i}^{I,J}\left(M\right) = \bigcap_{s > 0} \ x^{s}{\rm H}_{i}^{I,J}\left(M\right)\subseteq \bigcap_{s > 0} \ \mathfrak{a}^{s}{\rm H}_{i}^{I,J}\left(M\right) = 0$
\end{center}
for all $i > 0$.

Now the second part. Let $M$ be a linearly compact $R$-module. In this case, denote by $\mathfrak{M}$ a nuclear base of $M$. It follows from \cite[Property $4.7$]{macdonald} that $M = \varprojlim_{U \in \mathfrak{M}} \left(M/U\right)$, where the modules $M/U$ ($U \in \mathfrak{M}$) are Artinian $R$-modules. In virtue of Proposition \ref{comutatividade} we have that
\begin{center}
${\rm H}_{i}^{I,J}\left(M\right) = {\rm H}_{i}^{I,J}\left(\varprojlim_{U \in \mathfrak{M}} \left(M/U\right)\right) \cong \varprojlim_{U \in \mathfrak{M}} {\rm H}_{i}^{I,J}\left(M/U\right)$.
\end{center}
Since ${\rm Soc}\left(M/U\right)\subseteq {\rm Soc}\left(M\right)$, for all $U \in \mathfrak{M}$, and by hypothesis ${\rm Soc}\left(M\right) = 0$ it follows that ${\rm Soc}\left(M/U\right) = 0$ for all $U \in \mathfrak{M}$. Thus, by the part initial we have that ${\rm H}_{i}^{I,J}\left(M/U\right) = 0$ for all $i > 0$ and for all $U \in \mathfrak{M}$. Therefore, ${\rm H}_{i}^{I,J}\left(M\right) = 0$, for all $i > 0$, as required.
\end{proof}

{\rm The following result is a generalization, to the case of a pair of ideals, of \cite[Theorem $4.8$]{uma teoria de homologia local para modulo linearmente compacto}.}

\begin{thm} Let $M$ be a linearly compact $R$-module where $R$ is Noetherian ring, with ${\rm N}\dim\left(M\right) = d$. Then, we have that ${\rm H}_{i}^{I,J}\left(M\right) = 0$ for all $i > d$.
\end{thm}

\begin{proof}
We do the proof in two cases. First the case in that $M$ is Artinian $R$-module. We prove this by induction on $d = {\rm N}\dim\left(M\right)$. When $d = 0$, by \cite[Proposition $4.8$]{completamento I-adico} it follows that ${\rm H}_{i}^{\mathfrak{a}}\left(M\right) = 0$ for all $i > 0$ and for all $\mathfrak{a} \in \tilde{W}\left(I,J\right)$; since, by definition, we have that ${\rm H}_{i}^{I,J}\left(M\right) = \varprojlim_{\mathfrak{a} \in \tilde{W}\left(I,J\right)} {\rm H}_{i}^{\mathfrak{a}}\left(M\right)$ it follows that ${\rm H}_{i}^{I,J}\left(M\right) = 0$ for all $i > 0$. Suppose now that $d > 0$. According to the proof of \cite[Proposition $4.8$]{completamento I-adico} we may assume without loss of generality that there exists an $x \in \mathfrak{a}$ such that $xM = M$ for all $\mathfrak{a} \in \tilde{W}\left(I,J\right)$. Thus the short exact sequence of Artinian modules
\begin{center}
$0\rightarrow \left(0 :_{M} x\right)\rightarrow M\stackrel{.x}{\rightarrow} M\rightarrow 0$
\end{center}
gives rise to a long exact sequence, by Corollary \ref{sequencia exata de artinianos}
$$
\begin{array}{lll}
\ldots \rightarrow {\rm H}_{i}^{I,J}\left(\left(0 :_{M} x\right)\right)\rightarrow {\rm H}_{i}^{I,J}\left(M\right)\stackrel{.x}{\rightarrow} {\rm H}_{i}^{I,J}\left(M\right)\rightarrow {\rm H}_{i - 1}^{I,J}\left(\left(0 :_{M} x\right)\right)\rightarrow \ldots.
\end{array}
$$
According to \cite[Lemma $4.7$]{uma teoria de homologia local para modulo linearmente compacto} we have ${\rm N}\dim\left(\left(0 :_{M} x\right)\right) \leq {\rm N}\dim\left(M\right) - 1 = d - 1$. It then follows from the inductive hypothesis that ${\rm H}_{i}^{I,J}\left(\left(0 :_{M} x\right)\right) = 0$ for all $i > d - 1$. Hence ${\rm H}_{i}^{I,J}\left(M\right) \cong x{\rm H}_{i}^{I,J}\left(M\right)$ for all $i > d$, therefore $${\rm H}_{i}^{I,J}\left(M\right) \cong \bigcap_{s > 0} \ x^{s}{\rm H}_{i}^{I,J}\left(M\right)\subseteq \bigcap_{s > 0} \ \mathfrak{a}^{s}{\rm H}_{i}^{I,J}\left(M\right) = 0$$ \noindent by Proposition \ref{dual}, item $\left(i\right)$. This completes the inductive step.

Now the second case. Let $M$ be a linearly compact $R$-module. In this case, denote by $\mathfrak{M}$ a nuclear base of $M$. It follows from \cite[Property $4.7$]{macdonald} that $M = \varprojlim_{U \in \mathfrak{M}} \left(M/U\right)$, where the modules $M/U$ ($U \in \mathfrak{M}$) are Artinian $R$-modules. In virtue of Proposition \ref{comutatividade} we have that
\begin{center}
${\rm H}_{i}^{I,J}\left(M\right) = {\rm H}_{i}^{I,J}\left(\varprojlim_{U \in \mathfrak{M}} \left(M/U\right)\right) \cong \varprojlim_{U \in \mathfrak{M}} {\rm H}_{i}^{I,J}\left(M/U\right)$.
\end{center}
According to the proof of \cite[Theorem $4.8$]{uma teoria de homologia local para modulo linearmente compacto} we have that ${\rm N}\dim\left(M/U\right) \leq {\rm N}\dim\left(M\right)$. By the part initial we have that ${\rm H}_{i}^{I,J}\left(M/U\right) = 0$ for all $i > d$ and for all $U \in \mathfrak{M}$. Therefore, ${\rm H}_{i}^{I,J}\left(M\right) = 0$ for all $i > {\rm N}\dim\left(M\right)$. We conclude the proof.
\end{proof}

\begin{prop} Let $\left(R,\mathfrak{m}\right)$ be a local Noetherian ring and $M$ a non-zero semidiscrete linearly compact $R$-module. Thus, there exists an element $x \in \mathfrak{a}$, for some $\mathfrak{a} \in \tilde{W}\left(\mathfrak{m},J\right)$, such that $xM = M$ and $\left(0 :_{M} x\right) = 0$ if and only if ${\rm H}_{i}^{\mathfrak{m},J}\left(M\right) = 0$ for all $i \geq 0$.
\end{prop}

\begin{proof}
Let ${\rm H}_{i}^{\mathfrak{m},J}\left(M\right) = 0$ for all $i \geq 0$. Thus, by \cite[Remark $4.6$]{on formal local cohomology} we have ${\rm H}_{i}^{\mathfrak{a}}\left(M\right) = 0$ for all $i \geq 0$ and for some $\mathfrak{a} \in \tilde{W}\left(\mathfrak{m},J\right)$. By \cite[Corollary $2.5$]{completamento I-adico} we have $\mathfrak{a}M = M$; we obtain then that $xM = M$ for some $x \in \mathfrak{a}$. On the other hand, it follows from the short exact sequence of linearly compact $R$-modules $0\rightarrow \left(0 :_{M} x\right)\rightarrow M\stackrel{x}{\rightarrow} M\rightarrow 0$ that we have the long exact sequence, by Theorem \ref{sequencia exata de linearmente compacto}
\begin{center}
$\ldots \rightarrow {\rm H}_{i}^{\mathfrak{m},J}\left(\left(0 :_{M} x\right)\right)\rightarrow {\rm H}_{i}^{\mathfrak{m},J}\left(M\right)\stackrel{x}{\rightarrow} {\rm H}_{i}^{\mathfrak{m},J}\left(M\right)\rightarrow {\rm H}_{i - 1}^{\mathfrak{m},J}\left(\left(0 :_{M} x\right)\right)\rightarrow \ldots$
\end{center}
and so it follows that ${\rm H}_{i}^{\mathfrak{m},J}\left(\left(0 :_{M} x\right)\right) = 0$ for all $i \geq 0$. Thus, by \cite[Remark $4.6$]{on formal local cohomology} it follows that ${\rm H}_{i}^{\mathfrak{a}}\left(\left(0 :_{M} x\right)\right) = 0$ for some $\mathfrak{a} \in \tilde{W}\left(\mathfrak{m},J\right)$ and for all $i \geq 0$. Since $\left(0 :_{M} x\right)$ is Artinian by \cite[Corollary $1$]{zoschinger}, $\left(0 :_{M} x\right) = 0$ by \cite[Proposition $4.10$]{completamento I-adico}. Conversely, suppose that $xM = M$ and $\left(0 :_{M} x\right) = 0$, then for all $i \geq 0$
\begin{center}
${\rm H}_{i}^{\mathfrak{m},J}\left(M\right) = x{\rm H}_{i}^{\mathfrak{m},J}\left(M\right) = \bigcap_{s > 0} \ x^{s}{\rm H}_{i}^{\mathfrak{m},J}\left(M\right)\subseteq \bigcap_{s > 0} \ \mathfrak{a}^{s}{\rm H}_{i}^{\mathfrak{m},J}\left(M\right) = 0$
\end{center}
by Proposition \ref{dual}, item $\left(i\right)$.
\end{proof}

\begin{prop} Let $M$ be an Artinian $R$-module where $R$ is a local ring, with maximal ideal $\mathfrak{m}$, and Noetherian. Then the following conditions are equivalent:
\begin{itemize}
\item[$\left(i\right)$] ${\rm D}\left(M\right)$ is $\left(I,J\right)$-torsion $R$-module.
\item[$\left(ii\right)$] ${\rm H}_{i}^{I,J}\left(M\right) = 0$ for all integers $i > 0$.
\end{itemize}
\end{prop}

\begin{proof}
$\left(i\right)\Rightarrow \left(ii\right)$: From Corollary \ref{completo}, we may assume in this case that $\left(R,\mathfrak{m}\right)$ is a complete ring. Since $M$ is a Artinian $R$-module it follows, by \cite[Theorem $10.2.12$, item $\left(iii\right)$]{brodmann}, that ${\rm D}\left(M\right)$ is a Noetherian $R$-module. Therefore, ${\rm D}\left(M\right)$ is finitely generated. Since $\Gamma_{I,J}\left({\rm D}\left(M\right)\right) = {\rm D}\left(M\right)$ it follows, by \cite[Corollary $4.2$]{cohomologia local de um par de ideais}, that ${\rm H}^{i}_{I,J}\left({\rm D}\left(M\right)\right) = 0$, for all $i > 0$, $i \in \mathbb{Z}$. Therefore, by the Corollary \ref{Corollary2.3}, item $\left(i\right)$, we have ${\rm D}\left({\rm H}^{i}_{I,J}\left({\rm D}\left(M\right)\right)\right) = 0$, for all $0 < i \in \mathbb{Z}$. On the other hand, we have, by Proposition \ref{dual}, item $\left(ii\right)$, that ${\rm D}\left(H^{i}_{I,J}\left({\rm D}\left(M\right)\right)\right) \cong {\rm H}_{i}^{I,J}\left({\rm D}\left({\rm D}\left(M\right)\right)\right)$; and as, by \cite[Theorem $10.2.12$, item $\left(iii\right)$]{brodmann}, we have that ${\rm D}\left({\rm D}\left(M\right)\right) \cong M$ it follows that ${\rm H}_{i}^{I,J}\left(M\right) = 0$ for all $i > 0$.\newline
$\left(ii\right)\Rightarrow \left(i\right)$: From Corollary \ref{completo}, we may assume in this case that $\left(R,\mathfrak{m}\right)$ is a complete ring. Since ${\rm H}_{i}^{I,J}\left(M\right) = 0$ it follows that ${\rm D}\left({\rm H}_{i}^{I,J}\left(M\right)\right) = 0$; thus, we have that ${\rm H}^{i}_{I,J}\left({\rm D}\left(M\right)\right) = 0$, since ${\rm D}\left({\rm H}_{i}^{I,J}\left(M\right)\right) \cong {\rm H}^{i}_{I,J}\left({\rm D}\left(M\right)\right)$, by Corollary \ref{Corollary2.3}, item $\left(ii\right)$. Now, by the \cite[Theorem $10.2.12$, item $\left(iii\right)$]{brodmann} we have that ${\rm D}\left(M\right)$ is finitely generated $R$-module, because is Noetherian. Therefore, by \cite[Corollary $4.2$]{cohomologia local de um par de ideais}, it follows that ${\rm D}\left(M\right)$ is $\left(I,J\right)$-torsion $R$-module.
\end{proof}

{\rm The following result is a generalization, to the case of a pair of ideals, of \cite[Corollary $4.4$]{cohomologia local de um par de ideais}.}

\begin{thm} Let $M$ be a linearly compact module over a local ring, and Noetherian $R$. Suppose that $J \neq R$. Then ${\rm H}_{i}^{I,J}\left(M\right) = 0$ for any $i > \dim\left(R/J\right)$.
\end{thm}

\begin{proof}
We divide the proof in two cases. The first case, consider that $M$ is an Artinian $R$-module. From Corollary \ref{completo}, we may assume in this case that $\left(R,\mathfrak{m}\right)$ is a complete ring; by \cite[Theorem $10.2.12$, item $\left(iii\right)$]{brodmann} we have that ${\rm D}\left({\rm D}\left(M\right)\right) \cong M$ and ${\rm D}\left(M\right)$ is a Noetherian $R$-module. Thus, we have ${\rm D}\left(M\right)$ a finitely generated $R$-module. By the Proposition \ref{dual}, item $\left(ii\right)$, we have that
\begin{center}
${\rm D}\left({\rm H}^{i}_{I,J}\left({\rm D}\left(M\right)\right)\right) \cong {\rm H}_{i}^{I,J}\left({\rm D}\left({\rm D}\left(M\right)\right)\right) \cong {\rm H}_{i}^{I,J}\left(M\right)$.
\end{center}
We have, by \cite[Corollary $4.4$]{cohomologia local de um par de ideais}, that ${\rm H}^{i}_{I,J}\left({\rm D}\left(M\right)\right) = 0$, for any $i > \dim\left(R/J\right)$. Therefore, by the Corollary \ref{Corollary2.3}, item $\left(i\right)$, we have ${\rm D}\left({\rm H}^{i}_{I,J}\left({\rm D}\left(M\right)\right)\right) = 0$, for all $i > \dim\left(R/J\right)$. Since, as seen above, we have that ${\rm D}\left({\rm H}^{i}_{I,J}\left({\rm D}\left(M\right)\right)\right) \cong {\rm H}_{i}^{I,J}\left(M\right)$, it follows the result.

The second case, let $M$ be a linearly compact $R$-module. In this case, denote by $\mathfrak{M}$ a nuclear base of $M$. It follows from \cite[Property $4.7$]{macdonald} that $M = \varprojlim_{U \in \mathfrak{M}} \left(M/U\right)$, where the modules $M/U$ ($U \in \mathfrak{M}$) are Artinian $R$-modules. In virtue of Proposition \ref{comutatividade} we have that
\begin{center}
${\rm H}_{i}^{I,J}\left(M\right) = {\rm H}_{i}^{I,J}\left(\varprojlim_{U \in \mathfrak{M}} \left(M/U\right)\right) \cong \varprojlim_{U \in \mathfrak{M}} {\rm H}_{i}^{I,J}\left(M/U\right)$.
\end{center}
By the part initial we have that ${\rm H}_{i}^{I,J}\left(M/U\right) = 0$ for all $i > \dim\left(R/J\right)$ and for all $U \in \mathfrak{M}$. Therefore, ${\rm H}_{i}^{I,J}\left(M\right) = 0$ for all $i > \dim\left(R/J\right)$, as required.
\end{proof}

\begin{prop} Let $n$ be a nonnegative integer. Suppose that ${\rm H}^{i}_{I,J}\left(R\right) = 0$ for all $i > n$, where $R$ is Noetherian local ring. Then ${\rm H}_{i}^{I,J}\left(M\right) = 0$ for all $i > n$ and for any linearly compact $R$-module $M$, which is not necessarily finitely generated.
\end{prop}

\begin{proof}
Consider first the case in that $M$ is an Artinian $R$-module. From Corollary \ref{completo}, we may assume in this case that $\left(R,\mathfrak{m}\right)$ is a complete ring; by \cite[Theorem $10.2.12$, item $\left(iii\right)$]{brodmann} we have that ${\rm D}\left({\rm D}\left(M\right)\right) \cong M$ and ${\rm D}\left(M\right)$ is a Noetherian $R$-module. Thus, we have ${\rm D}\left(M\right)$ a finitely generated $R$-module. By the Proposition \ref{dual}, item $\left(ii\right)$, we have that
\begin{center}
${\rm D}\left({\rm H}^{i}_{I,J}\left({\rm D}\left(M\right)\right)\right) \cong {\rm H}_{i}^{I,J}\left({\rm D}\left({\rm D}\left(M\right)\right)\right) \cong {\rm H}_{i}^{I,J}\left(M\right)$.
\end{center}
By the \cite[Lemma $4.8$, item $\left(1\right)$]{cohomologia local de um par de ideais} we have that ${\rm H}^{i}_{I,J}\left({\rm D}\left(M\right)\right) = 0$, for all $i > n$. Therefore, by the Corollary \ref{Corollary2.3}, item $\left(i\right)$, it follows that ${\rm D}\left({\rm H}^{i}_{I,J}\left({\rm D}\left(M\right)\right)\right) = 0$ for all $i > n$ and then, it follows that ${\rm H}_{i}^{I,J}\left(M\right) = 0$ for all $i > n$, as required.

Now, let $M$ be a linearly compact $R$-module. In this case, denote by $\mathfrak{M}$ a nuclear base of $M$. It follows from \cite[Property $4.7$]{macdonald} that $M = \varprojlim_{U \in \mathfrak{M}} \left(M/U\right)$, where the modules $M/U$ ($U \in \mathfrak{M}$) are Artinian $R$-modules. In virtue of Proposition \ref{comutatividade} we have that
\begin{center}
${\rm H}_{i}^{I,J}\left(M\right) = {\rm H}_{i}^{I,J}\left(\varprojlim_{U \in \mathfrak{M}} \left(M/U\right)\right) \cong \varprojlim_{U \in \mathfrak{M}} {\rm H}_{i}^{I,J}\left(M/U\right)$.
\end{center}
By the part initial we have that ${\rm H}_{i}^{I,J}\left(M/U\right) = 0$ for all $i > n$ and for all $U \in \mathfrak{M}$. Therefore, ${\rm H}_{i}^{I,J}\left(M\right) = 0$ for all $i > n$.
\end{proof}

{\rm We have also the following results on the vanishing of local homology modules with respect to $\left(I,J\right)$, which are immediate consequences of results of \cite{cohomologia local de um par de ideais}.}

\begin{rem} {\rm Recall that the arithmetic rank of an ideal $I$, denoted by ${\rm ara}\left(I\right)$, is defined to be the least number of elements of $R$ required to generate an ideal which has the same radical as $I$.
\begin{itemize}
\item[$\left(i\right)$] Let $M$ be a finitely gerated $R$-module with $R$ be a local ring and Noetherian. Suppose that $J \neq R$. Then ${\rm H}_{i}^{I,J}\left({\rm D}\left(M\right)\right) = 0$ for any $i > \dim\left(M/JM\right)$. (immediate consequence of \cite[Theorem $4.3$]{cohomologia local de um par de ideais})
\item[$\left(ii\right)$] (immediate consequence of \cite[Theorem $4.7$]{cohomologia local de um par de ideais}) Let $M$ be a finitely generated $R$-module, where $R$ is Noetherian local ring. Then
\begin{itemize}
\item[$\left(1\right)$] ${\rm H}_{i}^{I,J}\left({\rm D}\left(M\right)\right) = 0$ for all integers $i > \dim\left(M\right)$.
\item[$\left(2\right)$] ${\rm H}_{i}^{I,J}\left({\rm D}\left(M\right)\right) = 0$ for all integers $i > \dim\left(M/JM\right) + 1$.
\end{itemize}
\item[$\left(iii\right)$] Let $M$ be an $R$-module, where $R$ is Noetherian local ring. Then for any integer $i > {\rm ara}\left(I\bar{R}\right)$, where $\bar{R} = R/\sqrt{J + {\rm Ann}_{R}\left(M\right)}$, we have that ${\rm H}_{i}^{I,J}\left({\rm D}\left(M\right)\right) = 0$. (immediate consequence of \cite[Proposition $4.11$]{cohomologia local de um par de ideais})
\end{itemize}}
\end{rem}

\section{The finiteness of co-associated primes}

\begin{defn} {\rm An $R$-module $M$ is called $I$-stable, where $I$ is ideal of ring $R$, if for each element $x \in I$, there is a positive integer $n$ such that $x^{t}M = x^{n}M$, for all $t \geq n$.}
\end{defn}

\begin{lem}\label{estavel} Let $M$, $N$ and $P$ $R$-modules. Suppose that we have the sequence of $R$-modules
\begin{center}
$0\rightarrow {\rm H}_{i}^{I,J}\left(M\right)\rightarrow {\rm H}_{i}^{I,J}\left(N\right)\rightarrow {\rm H}_{i}^{I,J}\left(P\right)\rightarrow 0$
\end{center}
being a short exact sequence for all $i \geq 0$. Then the module ${\rm H}_{i}^{I,J}\left(N\right)$ is $\mathfrak{a}$-stable if and only if the modules ${\rm H}_{i}^{I,J}\left(M\right)$, ${\rm H}_{i}^{I,J}\left(P\right)$ are $\mathfrak{a}$-stable, for any $\mathfrak{a} \in \tilde{W}\left(I,J\right)$.
\end{lem}

\begin{proof} $\left(\Rightarrow\right)$ Assume that ${\rm H}_{i}^{I,J}\left(N\right)$ is $\mathfrak{a}$-stable. Thus, for each element $x \in \mathfrak{a}$ there is a positive integer $n$ such that
\begin{center}
$x^{n}{\rm H}_{i}^{I,J}\left(N\right) = \bigcap_{t > 0} x^{t}{\rm H}_{i}^{I,J}\left(N\right)\subseteq \bigcap_{t > 0} \mathfrak{a}^{t}{\rm H}_{i}^{I,J}\left(N\right) = 0$
\end{center}
since ${\rm H}_{i}^{I,J}\left(N\right)$ is $\mathfrak{a}$-separated. As we have the sequence
\begin{center}
$0\rightarrow {\rm H}_{i}^{I,J}\left(M\right)\rightarrow {\rm H}_{i}^{I,J}\left(N\right)\rightarrow {\rm H}_{i}^{I,J}\left(P\right)\rightarrow 0$
\end{center}
exact it follows that $x^{n}{\rm H}_{i}^{I,J}\left(M\right) = x^{n}{\rm H}_{i}^{I,J}\left(P\right) = 0$ and thus, by definition, ${\rm H}_{i}^{I,J}\left(M\right)$ and ${\rm H}_{i}^{I,J}\left(P\right)$ are $\mathfrak{a}$-stable.

$\left(\Leftarrow\right)$ As ${\rm H}_{i}^{I,J}\left(M\right)$ and ${\rm H}_{i}^{I,J}\left(P\right)$ are $\mathfrak{a}$-stable, for each element $x \in \mathfrak{a}$, there is a positive integer $m$ such that $x^{m}{\rm H}_{i}^{I,J}\left(M\right) = x^{m}{\rm H}_{i}^{I,J}\left(P\right) = 0$. By \cite[Lemma $9.1.1$]{brodmann}, there is a positive integer $r$ such that $x^{r}{\rm H}_{i}^{I,J}\left(N\right) = 0$. Hence, ${\rm H}_{i}^{I,J}\left(N\right)$ is $\mathfrak{a}$-stable.
\end{proof}

Let $S$ be a multiplicative set of ring $R$. According to \cite{left derived functors} the co-localization of an $R$-module $M$ with respect to $S$ is the module $_{S}M = {\rm Hom}_{R}\left(S^{- 1}R,M\right)$; let $\mathfrak{p}$ be a prime of $R$ and $S = R \setminus \left\{\mathfrak{p}\right\}$, then instead of $_{S}M$ we write $_{\mathfrak{p}}M$. According to yet \cite{left derived functors}, for an $R$-module $M$ we have that {\it Co-support} of $M$ is the set ${\rm Cos}_{R}\left(M\right) = \left\{\mathfrak{p} \in {\rm Spec}\left(R\right) \;|\; _{\mathfrak{p}}M \neq 0\right\}$.

The following result is a generalization, to the case of a pair of ideals, of \cite[Corollary $3.11$]{left derived functors}.

\begin{prop}\label{co-localizacao nula} Let $S$ be a multiplicative set and $\mathfrak{a} \in \tilde{W}\left(I,J\right)$ an any ideal of $R$ such that $S \cap \mathfrak{a} \neq \emptyset$. Then $_{S}\left({\rm H}_{i}^{I,J}\left(M\right)\right) = 0$ for all $R$-module $M$ and for all $i \geq 0$.
\end{prop}

\begin{proof}
Since, by Proposition \ref{dual}, item $\left(i\right)$, we have that ${\rm H}_{i}^{I,J}\left(M\right)$ is a $\mathfrak{a}$-separated $R$-module for all $\mathfrak{a} \in \tilde{W}\left(I,J\right)$ it follows from \cite[Lemma $3.10$]{left derived functors} that $_{S}\left({\rm H}_{i}^{I,J}\left(M\right)\right) = 0$ for all $i \geq 0$.
\end{proof}

The following result is a generalization, to the case of a pair of ideals, of \cite[Proposition $3.13$]{left derived functors}.

\begin{thm}\label{iso} Let $M$ be a linearly compact $R$-module. Then $$_{S}\left({\rm H}_{i}^{I,J}\left(M\right)\right) \cong {\rm H}_{i}^{S^{- 1}I,S^{- 1}J}\left(_{S}M\right) \mbox{ for all } i \geq 0.$$
\end{thm}

\begin{proof}
From \cite[Corollary $2.25$]{rotman}, the co-localization functor $_{S}\left(-\right)$ preserves inverse limits, then
$$
\begin{array}{lll}
_{S}\left({\rm H}_{i}^{I,J}\left(M\right)\right)& = &_{S}\left(\varprojlim_{\mathfrak{a} \in \tilde{W}\left(I,J\right)} \varprojlim_{t \in \mathbb{N}} {\rm Tor}_{i}^{R}\left(R/\mathfrak{a}^{t},M\right)\right)\\
&\cong &\varprojlim_{\mathfrak{a} \in \tilde{W}\left(I,J\right)} \varprojlim_{t \in \mathbb{N}} \ _{S}\left({\rm Tor}_{i}^{R}\left(R/\mathfrak{a}^{t},M\right)\right).
\end{array}
$$

Now, we have by \cite[Lemma $3.12$]{left derived functors},
$$
\begin{array}{lll}
_{S}\left({\rm H}_{i}^{I,J}\left(M\right)\right)&\cong &\varprojlim_{\mathfrak{a} \in \tilde{W}\left(I,J\right)} \varprojlim_{t \in \mathbb{N}} {\rm Tor}_{i}^{S^{- 1}R}\left(S^{- 1}R/\left(S^{- 1}\mathfrak{a}\right)^{t},_{S}M\right)\\
&=& \varprojlim_{S^{- 1}\mathfrak{a} \in \tilde{W}\left(S^{- 1}I,S^{- 1}J\right)} {\rm H}_{i}^{S^{- 1}\mathfrak{a}}\left(_{S}M\right)\\
&=& {\rm H}_{i}^{S^{- 1}I,S^{- 1}J}\left(_{S}M\right)
\end{array}
$$

for all $i \geq 0$, as required.
\end{proof}

Let, according to \cite[Definition $1.5$]{cohomologia local de um par de ideais}, $W\left(I,J\right)$ denote the set of prime ideals $\mathfrak{p}$ of $R$ such that $I^{n}\subseteq J + \mathfrak{p}$ for some integer $n \geq 1$, i.e., we have the set $W\left(I,J\right) = \left\{\mathfrak{p} \in {\rm Spec}\left(R\right) \;|\; I^{n}\subseteq J + \mathfrak{p} \ \textrm{for some integer} \ n \geq 1\right\}$.

We have then the following corollary.

\begin{cor} Let $M$ be a linearly compact $R$-module. Then
\begin{center}
${\rm Cos}_{R}\left({\rm H}_{i}^{I,J}\left(M\right)\right)\subseteq {\rm Cos}_{R}\left(M\right) \cap W\left(I,J\right)$
\end{center}
for all $i \geq 0$.
\end{cor}

\begin{proof}
Let $\mathfrak{p} \in {\rm Cos}_{R}\left({\rm H}_{i}^{I,J}\left(M\right)\right)$. We have an isomorphism for all $i \geq 0$ by Theorem \ref{iso},
\begin{center}
$_{\mathfrak{p}}\left({\rm H}_{i}^{I,J}\left(M\right)\right) \cong {\rm H}_{i}^{S^{- 1}I,S^{- 1}J}\left(_{\mathfrak{p}}M\right)$
\end{center}
for all $i \geq 0$ where $S = R \setminus \left\{\mathfrak{p}\right\}$. It follows that ${\rm H}_{i}^{S^{- 1}I,S^{- 1}J}\left(_{\mathfrak{p}}M\right) \neq \left\{0\right\}$ and hence $_{\mathfrak{p}}M \neq \left\{0\right\}$. Moreover, by Proposition \ref{co-localizacao nula}, we have that $\mathfrak{p} \in V(\mathfrak{a})$ for all ${\mathfrak a}\in \tilde{W}\left(I,J\right)$. So, we have ${\mathfrak p}\in W\left(I,J\right)$. Therefore, $\mathfrak{p} \in {\rm Cos}_{R}\left(M\right) \cap W\left(I,J\right)$, as required.
\end{proof}

A prime ideal $\mathfrak{p}$ of ring $R$ is said to be a coassociated prime ideal of $M$ if there is an Artinian quotient $L$ of $M$ such that $\mathfrak{p} = \left(0 :_{R} L\right)$. The set of all coassociated prime ideals of $M$ is denoted by ${\rm Coass}_{R}\left(M\right)$.

\begin{thm}\label{finitude} Let $R$ be a Noetherian ring and $M$ an $\mathfrak{a}$-stable semidiscrete linearly compact $R$-module, for any $\mathfrak{a} \in \tilde{W}\left(I,J\right)$. Let $i$ be a non-negative integer. If, ${\rm H}_{j}^{I,J}\left(M\right)$ is $\mathfrak{a}$-stable, for all $j < i$, and $G$ is a closed $R$-submodule of ${\rm H}_{i}^{I,J}\left(M\right)$ such that ${\rm H}_{i}^{I,J}\left(M\right)/G$ is $\mathfrak{a}$-stable, then the set ${\rm Coass}_{R}\left(G\right)$ is finite.
\end{thm}

\begin{proof} We prove by induction on $i$.

When $i = 0$, we have that $G$ is a closed $R$-submodule of ${\rm H}_{0}^{I,J}\left(M\right)$. As $M$ is a linearly compact $R$-module, by Theorem \ref{linearmente compacto} we have that ${\rm H}_{0}^{I,J}\left(M\right)$ is linearly compact $R$-module. Also, as $M$ is a semidiscrete $R$-module, we have by the \cite[Lemma $2.8$]{co-associados finitos} that ${\rm Tor}_{0}^{R}\left(R/\mathfrak{a}^{t},M\right)$ is semidiscrete $R$-module. By the proof of the \cite[Theorem $3.3$]{co-associados finitos} we have that ${\rm H}_{0}^{\mathfrak{a}}\left(M\right) = \varprojlim_{t \in \mathbb{N}} {\rm Tor}_{0}^{R}\left(R/\mathfrak{a}^{t},M\right)$ is semidiscrete $R$-module. Therefore, we have that $\varprojlim_{\mathfrak{a} \in \tilde{W}\left(I,J\right)} {\rm H}_{0}^{\mathfrak{a}}\left(M\right) = {\rm H}_{0}^{I,J}\left(M\right)$ is semidiscrete $R$-module. As $G$ is a submodule of ${\rm H}_{0}^{I,J}\left(M\right)$, we have that $G$ is also semidiscrete linearly compact. By \cite[$1$, property L4]{zoschinger} we have that ${\rm Coass}_{R}\left(G\right)$ is a finite set.

Let $i > 0$. Being $R$ be a Noetherian ring, the ideal $\mathfrak{a}$ is finitely generated. By the hypothesis, $M$ is $\mathfrak{a}$-stable and then by definition there is a positive integer $n$ such that $\mathfrak{a}^{t}M = \mathfrak{a}^{n}M$, for all $t \geq n$. Set $K = \mathfrak{a}^{n}M$; then by \cite[$3$, property $3.14$]{macdonald} we have that $K$ is linearly compact $R$-module. Now the short exact sequence of linearly compact $R$-modules:
\begin{center}
$0\rightarrow K\rightarrow M\rightarrow M/K\rightarrow 0$,
\end{center}
provides us, by Theorem \ref{sequencia exata de linearmente compacto}, an long exact sequence:
\begin{center}
$\ldots \rightarrow {\rm H}_{i + 1}^{I,J}\left(M/K\right)\rightarrow {\rm H}_{i}^{I,J}\left(K\right)\stackrel{f}{\rightarrow} {\rm H}_{i}^{I,J}\left(M\right)\stackrel{g}{\rightarrow} {\rm H}_{i}^{I,J}\left(M/K\right)\rightarrow \ldots$. \ \ \ (*)
\end{center}
Note that $M/K$ is complete in $\mathfrak{a}$-adic topology. By \cite[Lemma $2.6$]{co-associados finitos} we have an isomorphism, for all $i > 0$, ${\rm Tor}_{i}^{R}\left(R,M/K\right) \cong {\rm H}_{i}^{\mathfrak{a}}\left(M/K\right)$. Thus, it follows that, $\varprojlim_{\mathfrak{a} \in \tilde{W}\left(I,J\right)} {\rm Tor}_{i}^{R}\left(R,M/K\right) \cong {\rm H}_{i}^{I,J}\left(M/K\right)$. As $M/K$ is $\mathfrak{a}$-stable, ${\rm Tor}_{i}^{R}\left(R,M/K\right)$ is $\mathfrak{a}$-stable, and thus $\varprojlim_{\mathfrak{a} \in \tilde{W}\left(I,J\right)} {\rm Tor}_{i}^{R}\left(R,M/K\right)$ is $\mathfrak{a}$-stable, so that ${\rm H}_{i}^{I,J}\left(M/K\right)$ is $\mathfrak{a}$-stable, for all $i > 0$. By the hypothesis, ${\rm H}_{j}^{I,J}\left(M\right)$ is $\mathfrak{a}$-stable, for all $j < i$. By the Lemma \ref{estavel}, we have ${\rm H}_{j}^{I,J}\left(K\right)$ is also $\mathfrak{a}$-stable, for all $j < i$.

We now prove that if $S$ is a closed submodule of ${\rm H}_{i}^{I,J}\left(K\right)$ such that ${\rm H}_{i}^{I,J}\left(K\right)/S$ is $\mathfrak{a}$-stable, then ${\rm Coass}_{R}\left(S\right)$ is finite. Proceeding analogously as in the proof of \cite[Theorem $3.3$]{co-associados finitos} for local homology modules of a module with respect to the pair of ideals $\left(I,J\right)$, we get that ${\rm Coass}_{R}\left(S\right)$ is finite.

We now have exact sequences induced from the exact sequence \ (*)
\begin{flushleft}
$0 \rightarrow {\rm H}_{i}^{I,J}\left(K\right)/f^{- 1}\left(G\right) \stackrel{\bar{f}}{\rightarrow} {\rm H}_{i}^{I,J}\left(M\right)/G$; \ $0 \rightarrow ff^{- 1}\left(G\right) \rightarrow G \rightarrow g\left(G\right) \rightarrow 0$.
\end{flushleft}
By the hypothesis, ${\rm H}_{i}^{I,J}\left(M\right)/G$ is $\mathfrak{a}$-stable, so that is also ${\rm H}_{i}^{I,J}\left(K\right)/f^{- 1}\left(G\right)$. Then, ${\rm Coass}_{R}\left(f^{- 1}\left(G\right)\right)$ is finite by the argument previous, so that we have ${\rm Coass}_{R}\left(ff^{- 1}\left(G\right)\right)$ is finite. Moreover, $g\left(G\right)$ is an submodule of the module ${\rm H}_{i}^{I,J}\left(M/K\right)$. Recall that ${\rm H}_{i}^{I,J}\left(M/K\right) \cong \varprojlim_{\mathfrak{a} \in \tilde{W}\left(I,J\right)} {\rm Tor}_{i}^{R}\left(R,M/K\right)$, and moreover we have that $\varprojlim_{\mathfrak{a} \in \tilde{W}\left(I,J\right)} {\rm Tor}_{i}^{R}\left(R,M/K\right)$ is semidiscrete linearly compact $R$-module, by the \cite[Lemma $2.8$]{co-associados finitos}, and thus ${\rm Coass}_{R}\left(g\left(G\right)\right)$ is finite. Finally, the finiteness of ${\rm Coass}_{R}\left(G\right)$ follows from the last short exact sequence, as required.    \end{proof}

\begin{cor} Let $M$ an $\mathfrak{a}$-stable semidiscrete linearly compact $R$-module, for all $\mathfrak{a} \in \tilde{W}\left(I,J\right)$. Let $i$ be a non-negative integer. If ${\rm H}_{j}^{I,J}\left(M\right)$ is $\mathfrak{a}$-stable for all $j < i$, then the set ${\rm Coass}_{R}\left({\rm H}_{i}^{I,J}\left(M\right)\right)$ is finite.
\end{cor}

\begin{proof} It follows from of Theorem \ref{finitude} by replacing $G$ with ${\rm H}_{i}^{I,J}\left(M\right)$.
\end{proof}

\end{document}